\documentclass{mathscan}

\RequirePackage[utf8]{inputenc}
\RequirePackage[T1]{fontenc}

\RequirePackage{lmodern}           
\RequirePackage[final]{microtype}  

\RequirePackage{amssymb}    
\RequirePackage{amsthm}     
\RequirePackage{thmtools}   
\RequirePackage{mathtools}  
\RequirePackage{mathrsfs}   

\RequirePackage[dvipsnames,svgnames,cmyk]{xcolor}  
\RequirePackage{graphicx}         
\graphicspath{{figures/}}
\RequirePackage{tikz}             
\RequirePackage{tikz-cd}          
\usepackage{tikz-3dplot}

\usepackage{tkz-euclide}
\usetikzlibrary{calc,through}
\usetikzlibrary{intersections}

\usepackage[font=small]{subcaption} 
\usepackage{times,mathptmx}

\usepackage{pstricks-add}
\usepackage{pst-func}

\usepackage{enumitem}

\RequirePackage{xspace}         
\RequirePackage{textcomp}       
\RequirePackage{float} 

\RequirePackage{varioref}
\RequirePackage{hyperref}
\RequirePackage[nameinlink, capitalize, noabbrev]{cleveref}

\declaretheorem[numberwithin = section]{theorem}
\declaretheorem[sibling = theorem]{corollary}
\declaretheorem[sibling = theorem]{lemma}
\declaretheorem[sibling = theorem]{proposition}

\declaretheorem[sibling = theorem]{definition}
\declaretheorem[sibling = theorem]{example}

\declaretheorem[sibling = theorem]{remark}

\newcounter{introthmnum}
\newenvironment{introthm}[1]{\stepcounter{introthmnum}
    \newtheorem*{introthmTemp:\arabic{introthmnum}}{#1}\begin{introthmTemp:\arabic{introthmnum}}}{\end{introthmTemp:\arabic{introthmnum}}}

\DeclarePairedDelimiter{\set}{\lbrace}{\rbrace}        
\newcommand{\N}{\mathbb{N}}    
\newcommand{\Z}{\mathbb{Z}}    
\newcommand{\Q}{\mathbb{Q}}    
\newcommand{\R}{\mathbb{R}}    
\newcommand{\C}{\mathbb{C}}    
\renewcommand{\P}{\mathbb{P}}  
\newcommand{\F}{\mathcal{F}}

\DeclareMathOperator{\Hom}{Hom}
\DeclareMathOperator{\im}{im}
\DeclareMathOperator{\val}{val}
\DeclareMathOperator{\Trop}{Trop}
\DeclareMathOperator{\trop}{trop}

\DeclareMathOperator{\tm}{TM}
\DeclareMathOperator{\ctm}{CTM}

\DeclareMathOperator{\tropdiv}{tropdiv}
\DeclareMathOperator{\Pic}{Pic}

\newcommand{\Xbar}{{\overline{X}}}

\newcommand{\plane}[1]{\P_{#1}^2}
\newcommand{\curve}{C}
\newcommand{\varp}{\mathbf{p}}
\newcommand{\varx}{\mathbf{x}}

\newcommand{\torus}{\mathbf{T}}
\newcommand{\intTor}[1]{\torus_{#1}}
\newcommand{\torusCoords}[1]{(K^*)^{#1}}

\newcommand{\Osheaf}{\mathcal{O}}

\newcommand{\arrCur}{\mathcal{B}}
\newcommand{\arrCompCur}{X_{\arrCur}}
\newcommand{\arrCompCurBar}{\overline{X}_{\arrCur}}
\newcommand{\arrCurLines}{\mathcal{B}_L}
\newcommand{\arrCompCurLines}{\mathcal{B}_L}
\newcommand{\arrCompCurDeletion}[1]{X_{\arrCur \setminus \set{#1}}}

\newcommand{\cluster}{\mathcal{C}}

\newcommand{\variety}{X}
\newcommand{\varietyReal}{\variety(\R)}
\newcommand{\varietyComplex}{\variety(\C)}

\newcommand{\wtil}[1]{\widetilde{#1}}
\renewcommand{\div}{\mathrm{div}}

\newcommand{\pbf}{\mathbf{p}}
\newcommand{\groundset}{\mathcal{A}}
\newcommand{\points}{\mathcal{P}}

\newcommand{\arroid}{\mathbf{A}}
\newcommand{\arroidDeletion}[1]{\arroid\setminus #1}

\newcommand{\divisor}{\mathbf{D}}

\DeclareMathOperator{\cone}{cone}
\newcommand{\supface}{\succ}

\newcommand{\subface}{\prec}

\newcommand{\mulSet}{m}
\newcommand{\mul}[1]{m_{#1}}
\newcommand{\Star}[2]{{{#1}^{#2}}}
\newcommand{\lattice}[1]{L_\Z (#1)}
\newcommand{\sign}[2]{\mathrm{sign}(#1,#2)}
\newcommand{\mincone}{\mathbf{0}}

\begin{document}

\title{Tropicalization of curve arrangement complements and arroids}

\author{Edvard Aksnes}
\thanks{This is the accepted version of the article published in \emph{Mathematica Scandinavica} \textbf{131} (2025), no.~2. DOI: \href{https://doi.org/10.7146/math.scand.a-156674}{10.7146/math.scand.a-156674}}
\address{Department of Mathematics, University of Oslo, Oslo, Norway.}
\email{\href{edvardak@math.uio.no}{edvardak@math.uio.no}}

\begin{abstract}
    We define \emph{arroids} as an abstract axiom set encoding the intersection properties of arrangements of curves. The tropicalization of the complement of an arrangement of curves meeting pairwise transversely is shown to be determined by the associated arroid. We give conditions for when the cohomology of the complement of an arrangement is computable using tropical cohomology, and we give criteria for when the complement is a maximal variety in terms of tropical geometry.
\end{abstract}
\maketitle

\setcounter{tocdepth}{1}
\tableofcontents

\section{Introduction}
Drawing inspiration from matroids, which abstractly axiomatize arrangements of hyperplanes, we define \emph{arroids}, which provide a possible abstract axiom set for the incidence geometry of arrangements of curves in the plane. To any arrangement of curves, one may associate an arroid. An arroid $\arroid$ consists of an underlying set $\groundset$ where each element $i$ is equipped with a degree $d_i$, along with a multiset $\points$ of subsets of $\groundset$. Each set $\pbf \in \points$ is equipped with a multiplicity function $\mul{\pbf}\colon \pbf^2 \to \Z$, and the multiset $\points$ must satisfy a Bézout condition in terms of the multiplicity functions. 

When all the multiplicity functions of an arroid are constant taking value one, the arroid is said to be \emph{transversal}. We construct a fan associated to each transversal arroid, and the following theorem shows that such a fan is a tropical variety, i.e. satisfies the balancing condition of tropical geometry, see e.g. \cite{BIMS14,MaclaganSturmfels} for definitions. 
\begin{introthm}{\cref{theorem:arroid_fan_tropical_variety}}
    For each transversal arroid $\arroid$, there is a fan $\Sigma_\arroid$, called the \emph{fan of $\arroid$}, which is a balanced tropical variety.
\end{introthm}
Using transversal arroids, we proceed to study the tropicalization of the complements of certain types of arrangements of curves. An arrangement of curves is \emph{very affine} if it contains at least three lines intersecting generically, and \emph{transverse} if all curves of the arrangement intersect pairwise transversely. In \cref{section:arroid_tropicalization}, we show that for a transverse very affine arrangement of curves, the tropicalization of the complement is computed by the arroid fan.
\begin{introthm}{\cref{theorem:transverse_very_affine_curve_fan}}
    Let $\arrCur$ be a transverse very affine arrangement of curves in the plane $\plane{K}$. Then the tropicalization $\trop(\arrCompCur)$ of the complement is supported on the fan of the associated transversal arroid $\arroid_\arrCur$.
\end{introthm}
For arrangements of lines, \cref{theorem:transverse_very_affine_curve_fan} recovers that the tropicalization of the complement is computed using the rank three matroid of the arrangement (see e.g. \cite[Theorem 4.1.11]{MaclaganSturmfels}), using the Ardila--Klivans fan structure \cite{ArdilaKlivans}.
The difficulty in generalizing beyond the transverse case lies primarily in understanding the resolution of singularities that arise when higher order intersections are allowed in the arrangement, reflecting an earlier observation found in \cite[p.20]{CuetoGeomTrop}.

Next, we turn to relating the cohomology of the complement of a very affine transverse arrangement of curves to the tropical cohomology of the fan of its associated transversal arroid. For a reminder on tropical cohomology, see \cref{section:tropical_co_homology}. 
For line arrangements and their corresponding matroids, an isomorphism between the matroid Orlik--Solomon algebra \cite{OSalg}, computing cohomology of the complement using only the intersection properties recorded by the matroid, and tropical cohomology of the matroid fan, was described by Zharkov \cite{Zharkov}. 
We consider transverse very affine arrangements $\arrCur$ of non-singular rational curves in $\plane{\C}$, i.e. of lines and conics, such that that no intersection point of the arrangement contains exactly the same curves. Such an arrangement will be called \emph{simple}. 

In \cref{prop:wunderschon_arrangement}, we show that the complement of a simple arrangement is \emph{wunderschön} in the sense of \cite[Definition 1.2]{AAPS23}, which is in this context primarily a restriction on its mixed Hodge structure. 
In light of \cite[Theorem 6.1]{AAPS23}, this implies that the complement of simple arrangements are \emph{cohomologically tropical} i.e. its rational cohomology can be computed using the $\Q$-coefficient tropical cohomology of its tropicalization, if and only if the corresponding arroid fan is a tropical homology manifold.
\begin{introthm}{\cref{theorem:cohom_trop}}
    Let $\arrCompCur$ be the complement of a simple arrangement $\arrCur$. Then $\arrCompCur$ is cohomologically tropical if and only if the corresponding arroid fan $\Sigma_{\arroid_\arrCur}$ is uniquely balanced along each of its rays.
\end{introthm}
This result follows from equivalent conditions for an arroid fan to be a tropical homology manifold given in \cref{theorem:arroid_THM_is_local}, and we study which conditions this imposes on curve arrangements in \cref{section:unique_balancing_rays_arrangement}.

Using \cref{theorem:cohom_trop}, we study the question of \emph{maximality} for a real arrangement and its complexification.
Let $\variety$ be a complex variety defined over $\R$, with $\varietyReal$ its set of real points and $\varietyComplex$ its set of complex points. The \emph{Smith-Thom inequality} gives bounds for the sum of the $\Z/2\Z$-Betti numbers as follows,
\begin{equation}\label{eq:smith-thom}
b_\bullet(\varietyReal) \coloneqq \sum_{i\geq 0} b_i(\varietyReal) \leq \sum_{i\geq 0} b_i(\varietyComplex) \eqqcolon b_\bullet(\varietyComplex),
\end{equation}
and the variety is \emph{maximal} if equality is achieved. In \cite{BrugalleSchaffhauser}, varieties with torsion-free cohomology satisfying the stronger inequalities 
$$b_i(\varietyReal) \leq \sum_j h^{i,j}(\varietyComplex)$$ 
in terms of the Hodge numbers of their complex parts are called \emph{sub-Hodge expressive}. In \cite{RenaudineauShaw}, Renaudineau and Shaw studied real algebraic hypersurfaces near the tropical limit, and gave bounds for Betti numbers in terms of tropical homology. Recently, Ambrosi and Manzaroli \cite{AmbrosiManzaroli} study the central fiber of a totally real semistable degeneration over a curve. They give three conditions on the components of the central fiber for each of the nearby fibers to be sub-Hodge expressive. For each open component $\variety$ of the central fiber, the conditions are the following:
\begin{enumerate}[label=(\alph*)]
    \item $H^i(\varietyReal;\Z/2\Z)=0$ for all $i\geq 1$,
    \item $\variety$ is a maximal variety, and
    \item the mixed Hodge structure on $H^i (\varietyComplex; \Q)$ is pure of type $(i,i)$ and integer cohomology is torsion free, for all degrees $i$.
\end{enumerate}
In light of these conditions, in \cref{theorem:THM_maximal} we give conditions for the complement of a simple arrangement of real curves in the plane to be maximal using its tropicalization. Moreover, using the wunderschön and cohomologically tropical properties, we give the following concrete construction of varieties satisfying the above conditions.
\begin{introthm}{\cref{theorem:amb_manz_examples}}
    Let $\arrCur$ be a simple arrangement of real curves in $\plane{\C}$, with all intersection points being real, and such that the tropicalization $\Trop(\arrCompCur)$, which is supported on the arroid fan $\Sigma_{\arroid_\arrCur}$, is a tropical homology manifold. Then the following four properties are satisfied:
    \begin{enumerate}[label=(\alph*)]
        \item $H^i (\arrCompCur(\R);\Z/2\Z)=0$ for $i\geq 1$,
        \item $\arrCompCur$ is a maximal variety, 
        \item the mixed Hodge structure on $H^i(\arrCompCur(\C);\Q)$ is pure of type $(i,i)$ with cohomology groups $H^i(\arrCompCur(\C);\Z)$ torsion-free for $i\geq 1$, and
        \item $\dim_\Q H^i(\arrCompCur(\C);\Q) = \sum_j \dim_\Q H^{i,j}(\Sigma_{\arroid_\arrCur})$ for each $i\geq 0$.
    \end{enumerate}
\end{introthm}
This theorem is illustrated by providing an infinite family of maximal surfaces in \cref{ex:inf_family}, which give examples of the types of variety required in \cite{AmbrosiManzaroli}, using conditions (a), (b) and (c). Condition (d) can be compared to the bounds for Betti numbers given in \cite{RenaudineauShaw}.

The paper is structured as follows. In \cref{section:preliminaries}, we recall the notions of geometric tropicalization, tropical modifications, and tropical cohomology. In \cref{section:curve_arr}, we illustrate how geometric tropicalization can be used to compute the tropicalization of the complement of an arrangement of curves in the plane. In \cref{section:arroids}, we introduce arroids as an abstract generalization of arrangements of curves, define fans associated to arroids which are \emph{transversal} (see \cref{def:transversal_arroid}), and relate the tropicalization of the complement of an arrangement to the fan of the associated arroid in the transversal case. Next, in \cref{section:THM_arroid}, we study which arroid fans are tropical homology manifolds, and give some conditions for the arroid fan of an arrangement of lines and conics to be a tropical homology manifold. Finally, in \cref{section:maximal}, use the concepts developed in the rest of the paper to study the maximality of arrangements of lines and conics.

\subsection*{Acknowledgments}
I am very grateful to Kris Shaw for many discussions and suggestions. I also thank the anonymous referee for insightful comments and suggestions. This research was supported by the Trond Mohn Foundation project ``Algebraic and Topological Cycles in Complex and Tropical Geometries''.

\section{Preliminaries}\label{section:preliminaries}
In this section, we recall certain notions in toric and tropical geometry. For the remainder of this paper, let $K$ be a trivially valued algebraically closed field.

\subsection{The intrinsic torus} We recall the definition and properties of the intrinsic torus, following \cite{HackingKeelTevelev,SturmfelsTevelev,Tevelev}.
For $X$ a variety, let $\Osheaf(X)$ be its coordinate ring, with group of units $\Osheaf(X)^*$. The group $\Osheaf(X)^* /K^*$ is free abelian of finite rank, and the torus $\intTor{X} \coloneqq \Hom(\Osheaf(X)^* /K^*, K^*)$ is called the \emph{intrinsic torus} of $X$, with lattice of characters $M = \Osheaf(X)^* /K^*$. Choosing a splitting of the quotient sequence for $M$, or equivalently a set of generators, defines a map $X \to \intTor{X}$, and we say that $X$ is \emph{very affine} is this gives a closed embedding. Furthermore, any closed embedding of a variety into an algebraic torus factors through an embedding to the intrinsic torus, composed with a monomial map. See \cite[Section 6.4]{MaclaganSturmfels} for details.

\subsection{Geometric tropicalization}\label{sec:geom_trop} We briefly recall the \emph{geometric tropicalization} approach to tropicalization. Let $\torus$ be an algebraic torus over the field $K$, with $M$ its lattice of characters and $N \coloneqq \Hom_\Z(M,\Z)$ its dual lattice. Let $X\subset \torus$ be a subvariety and $X'$ a normal, $\Q$-factorial variety birational to $X$. Any divisor $D$ on $X'$ induces a valuation $\val_D \colon K(X) \to \Z$ on the field of fractions $K(X)$ of $X$. Such a valuation is called a \emph{divisorial valuation} on $K(X)$, and it defines a vector $[\val_D]\in N_\Q$ by restricting a character $m\in M$ to a rational function $m|_X \in K(X)$ and evaluating it using $\val_D$.

The following proposition, originally shown in \cite[p. 176]{HackingKeelTevelev}, with an alternate proof given in \cite[Theorem 2.4]{SturmfelsTevelev}, characterizes the tropicalization of $X$ using divisorial valuations. 

\begin{proposition}\label{all_val}
    The tropicalization $\Trop(X)$ is equal to the closure of the subset 
    $$\set{c [\val_D] \mid c\in \R_{\geq 0},\; \val_D \text{ a divisorial valuation on } K(X)} \subseteq N_\R.$$ 
\end{proposition}

For any compactification $\Xbar$ of $X\subseteq \torus$ by a simple normal crossing divisor $\partial X \coloneqq \Xbar \setminus X$ with irreducible components $D_1,\dots, D_m$, Hacking, Keel and Tevelev give an explicit construction of $\Trop(X)$ as follows (see \cite[Theorem 2.3]{HackingKeelTevelev}). 

Let $\Delta_{\partial X}$ be the \emph{boundary complex} of $(\Xbar,\partial X)$, i.e. a simplicial complex with vertices $\set{1,\dots,m}$ containing the simplex $\set{i_1,\dots, i_k}$ if and only if $D_{i_1} \cap \dots \cap D_{i_k}$ is non-empty.  The $m$ divisorial valuations $[\val_{D_1}],\dots, [\val_{D_m}]$ give rays in $N_\R$, and for each simplex $\sigma$ of $\Delta$ we form the cone $[\sigma] \coloneqq \mathrm{cone}([\val_{D_i}] \mid i \in \sigma)$ in $N_\R$.
\begin{theorem}[{\cite[Theorem 2.3]{HackingKeelTevelev}}]\label{theorem:snc_trop}
    For $X\subseteq \torus$ compactified to $\Xbar$ by a simple normal crossing divisor $D$, we have
    $$\Trop(X)=\cup_{\sigma \in \Delta} [\sigma].$$
\end{theorem}
Under the assumption of working over a field of characteristic zero, the condition that the divisor $D$ is simple normal crossing can be weakened to being merely \emph{combinatorial normal crossing}, i.e. such that the intersection of any $l$ irreducible components is of codimension $l$.
\begin{theorem}[{\cite[Theorem 2.8]{CuetoGeomTrop}}]\label{cnc_trop}
    For $K$ a field of characteristic zero and $X\subseteq \torus$ compactified to $\Xbar$ by a combinatorial normal crossing divisor $D$ we have
    $$\Trop(X)=\cup_{\sigma \in \Delta} [\sigma].$$
\end{theorem}
\begin{remark}\label{cnc_no_char}
    The varieties considered in the present paper are surfaces, which admit resolutions of singularities (see \cite[\href{https://stacks.math.columbia.edu/tag/0BIC}{Tag 0BIC}]{stacks-project}), and as such the proofs given in \cite[Theorem 2.8]{CuetoGeomTrop} carry through even without assuming that the characteristic of the field $K$ is zero. 
\end{remark}

\subsection{Tropical modifications}\label{section:trop_modif}
We now recall the definition of tropical varieties in the case of fans, and then generalities on tropical divisors and tropical modifications, see \cite{MikhalkinRauBook} and \cite{AllermannRau} for details. 

\begin{definition}[{\cite[Definition 3.3.1.]{MaclaganSturmfels}, \cite[Definition 5.7]{BIMS14}}]
    A pure rational fan $\Delta$ in $\R^n$ of dimension $d$ is a \emph{tropical variety} if it satisfies the following \emph{balancing condition}. First, the top-dimensional cones $\sigma \in \Delta^d$ are equipped with integer weights $\omega(\sigma)$, giving a \emph{weight function} $\omega$. For each cone $\tau$ of codimension one, let $L$ be the subspace parallel to $\tau$, and since $\tau$ is rational, $L\cap \Z^n$ is a free lattice, with a quotient $N(\tau) = \Z^n/(L\cap \Z^n)$. Each top-dimensional cone $\sigma$ containing $\tau$ corresponds to a ray in $N(\tau)\otimes_\Z \R$, and let $v_\sigma$ be the first lattice point on this ray. The fan $\Delta$ is \emph{balanced at $\tau$} if 
    $$\sum \omega(\sigma) v_\sigma = 0,$$
    and $\Delta$ is \emph{balanced} if it is balanced for all codimension one cones.
\end{definition}

\begin{definition}
    Let $\Delta'$ be a tropical fan of dimension $d$ in $\R^n$ with weight function $\omega$.
    A \emph{tropical rational function} $\phi\colon \Delta \to \R$ is a function so that $\Delta'$ can be subdivided into $\Delta$ such that, for each cone $\sigma$ of $\Delta$, the restricted function $\phi\mid_\sigma$ is integer affine, i.e. there is some $b_\sigma\in \R$ and integer matrix $A_\sigma = (a_1, \dots, a_n) \in \Z^{1\times n}$ such that $\phi\mid_\sigma (\varx) = A_\sigma \cdot \varx + b_\sigma$. 
\end{definition}

For each cone $\gamma$ of $\Delta$, let $\lattice{\gamma}\subseteq N$ be the (saturated) lattice parallel to $\sigma$. Each facet $\sigma$ of $\Delta$ is equipped with a weight $\omega(\sigma)$. Since $\Delta$ is tropical, for each codimension one cone $\tau$, and each facet $\sigma$ containing $\tau$ (using notation $\tau \prec \sigma$) one may find a vector $v_{\sigma/\tau}\in L(\sigma)$ such that $\lattice{\sigma} = \lattice{\tau} + \Z v_{\sigma/\tau}$ and moreover $\sum_{\tau \prec \sigma} \omega(\sigma) v_{\sigma/\tau}=0$. 
\begin{definition}
    The \emph{divisor $\tropdiv(\phi)$ of $\phi$ on $\Delta$} is the weighted polyhedral complex whose maximal cones are the codimension one cones $\tau$ such that $\omega_\phi(\tau)\neq 0$, where $\omega_\phi(\tau)$ is the weight $\omega_\phi(\tau) \coloneqq \sum_{\tau \prec \sigma} \omega(\sigma) A_{\sigma}\cdot v_{\sigma/\tau}$.
\end{definition}

We now recall some notions for modifing tropical fans, namely two variations on the idea of a modification of a tropical fan, which will be used in later sections.
\begin{definition}
    The \emph{tropical modification} $\tm(\Delta,\tropdiv(\phi))$ of $\Delta$ along the tropical rational function $\phi$ is a fan in $\R^{n+1}$, with the following cones.
    \begin{itemize}
        \item For each cone $\sigma$ of $\Delta$, let $\wtil{\sigma} \coloneqq \set{(\varx,\phi(\varx))\in \R^{n+1}  \mid \varx \in \sigma}$, and define a weight $\wtil{\omega}(\wtil{\sigma})=\omega(\sigma)$. 
        \item For each cone $\tau$ of $\tropdiv(\phi)$, let $\tau_\geq = \set{ (\varx,y)\in \R^{n+1} \mid x\in \tau, \; y \geq \phi(\varx)}$, with weight $\wtil{\omega}(\tau_\geq)=\omega_\phi(\tau)$.
    \end{itemize}
    The underlying set of $\tm(\Delta, \tropdiv(\phi))$ is the union 
    $$\tm(\Delta, \tropdiv(\phi)) = (\cup_{\sigma \in \Delta} \wtil{\sigma}) \cup (\cup_{\tau \in \tropdiv(\phi)} \tau_\geq).$$
    Using the weights defined above, $\tm(\Delta, \tropdiv(\phi))$ is a balanced polyhedral complex.
\end{definition}
\begin{definition}
    The \emph{closed tropical modification} $\ctm(\Delta,\tropdiv(\phi))$ of $\Delta$ along the tropical rational function $\phi$ is a polyhedral complex in $\R^{n}\times (\R\cup \set{\infty})$, with the following polyhedra:
    \begin{itemize}
        \item a cone $\wtil{\sigma}$ for each $\sigma \in \Delta$, 
        \item a polyhedron $\tau_\geq$ for each cone $\tau$ of $\tropdiv(\phi)$, and
        \item for each cone $\tau$ of $\tropdiv(\phi)$, the corresponding ``face at infinity'' $\tau_\infty$ is the polyhedron $\tau_\infty = \set{ (\varx,\infty )\in \R^n \times (\R\cup \set{\infty}) \mid x\in \tau}$. 
    \end{itemize}
    The closed tropical modification is the polyhedrally-decomposed set
    $$\ctm(\Delta, \tropdiv(\phi)) = (\cup_{\sigma \in \Delta} \wtil{\sigma}) \cup (\cup_{\tau \in \tropdiv(\phi)} \tau_\geq) \cup (\cup_{\tau \in \tropdiv(\phi)} \tau_\infty),$$
    where again the same weights make $\ctm(\Delta, \tropdiv(\phi))$ a balanced polyhedral complex.
\end{definition}


\subsection{Tropical (co)homology}\label{section:tropical_co_homology}
Tropical homology and cohomology, as introduced in \cite{IKMZ}, is an invariant of a rational polyhedral space. While there are several different approaches, see for instance \cite{Lefschetz11,AminiPiquerezFans,AminiPiquerezHodge,TPDspace,GrossShokriehSheaf,JSS19,AAPS23}, here we briefly recall the definition of tropical (co)homology as most conveniently defined in the case of fans.

Let $\Sigma$ be a $d$-dimensional fan in $N$, with unique minimal $0$-dimensional cone denoted $\mincone$. Recall from \cref{section:trop_modif} that $\lattice{\sigma}\subseteq N$ is the (saturated) lattice parallel to $\sigma$. For each $p=1,\dots,d$, the \emph{$p$-th multi-tangent space} is the vector subspace
$$\F_p (\sigma) \coloneqq \sum_{\sigma \subface \gamma} \bigwedge^p \lattice{\gamma} \otimes_\Z \Q \subseteq \bigwedge^p N \otimes_\Z \Q,$$
where the sum is taken over all cones $\gamma$ of $\Sigma$ containing $\sigma$ as a face. Moreover, for $\tau$ a face of the cone $\sigma$, there is an inclusion $\iota_{\tau \subface \sigma} \colon \F_p (\sigma) \to \F_p (\tau)$ as vector subspaces.

In the case of a fan, the \emph{tropical cohomology} groups are given by $$H^{p,0}(\Sigma) = \F^p (\mincone) \coloneqq \Hom_\Q(\F_p (\mincone),\Q),$$
where $\mincone$ is the unique minimal $0$-dimensional cone of $\Sigma$, for $p=0,\dots,d$. Similarly, the \emph{tropical homology} groups defined by $H_{p,0}(\Sigma) = \F_p(\mincone)$.

Equip each cone $\sigma$ with an orientation, and for $\tau$ a face of $\sigma$, let $\sign{\tau}{\sigma}$ be $1$ if the chosen orientations of $\tau$ and $\sigma$ are compatible, and $-1$ otherwise. For each $p=1,\dots,d$, we may define the \emph{$p$-th tropical Borel--Moore complex} $C_{p,\bullet}^{BM}(\Sigma)$ as follows
\[\begin{tikzcd}[column sep=8pt]
	0 & {\oplus_{\alpha \in \Sigma_d} \F_p(\alpha)} & {\oplus_{\beta \in \Sigma_{d-1}} \F_p(\beta)} & \cdots & {\oplus_{\rho \in \Sigma_1} \F_p(\rho)} & {\F_p(\mincone)} & {0.}
	\arrow["{\partial_1}", from=1-5, to=1-6]
	\arrow[from=1-6, to=1-7]
	\arrow["{\partial_d}", from=1-2, to=1-3]
	\arrow["{\partial_{d-1}}", from=1-3, to=1-4]
	\arrow["{\partial_2}", from=1-4, to=1-5]
	\arrow[from=1-1, to=1-2]
\end{tikzcd}\]
The differential $\partial_k$ is defined as the sum of its components $(\partial_k)_{\gamma,\delta} \colon \F_p(\gamma) \to \F_p(\delta)$, which is given by $\sign{\delta}{\gamma}\cdot \iota_{\delta \subface \gamma}$ if $\delta$ is a face of $\gamma$, and $0$ otherwise. 
\begin{definition}
    The \emph{tropical Borel--Moore homology} group $H_{p,q}^{BM}(\Sigma)$ is the $q$-th homology group of the complex $C_{p,\bullet}^{BM}(\Sigma)$.
\end{definition}

For any tropical fan $\Sigma$ of dimension $d$, the weight function $\omega$ gives rise to a \emph{fundamental class} $[\Sigma,\omega] \in H_{d,d}^{BM}(\Sigma)$, and there are cap product maps $\frown [\Sigma, \omega] \colon H^{p,q}(\Sigma) \to H_{d-p,d-q}^{BM}(\Sigma)$. 
\begin{definition} Let $\Sigma$ be a tropical fan with fundamental class $[\Sigma, \omega]\in H_{d,d}^{BM}(\Sigma)$.
    \begin{itemize}
        \item If the maps $\frown [\Sigma, \omega] \colon H^{p,q}(\Sigma) \to H_{d-p,d-q}^{BM}(\Sigma)$ are isomorphisms for all $p$ and $q$, $\Sigma$ is said to satisfy \emph{tropical Poincaré duality}
        \item If each reduced star $\Star{\Sigma}{\gamma}$ of $\Sigma$ at a cone $\gamma$ satisfies tropical Poincaré duality, including $\Star{\Sigma}{\mincone}=\Sigma$, then $\Sigma$ is called a \emph{tropical homology manifold}. 
    \end{itemize}
\end{definition}
 Recall that \emph{reduced star} $\Star{\Sigma}{\gamma}$ of $\Sigma$ at a cone $\gamma$ is the fan in $N/\lattice{\gamma}$ consisting of the projection of the cones $\delta$ of $\Sigma$ containing $\gamma$. Note that this is sometimes called simply the \emph{star} in the context of toric geometry, see e.g. \cite[Section 3.1]{Fulton}, we make the distinction for compatibility with \cite{TPDspace}.

\section{Tropicalizing complements of arrangements of curves}\label{section:curve_arr}
We now turn to arrangements of curves $\arrCur \coloneqq \set{L_0,L_1,L_2,\curve_1,\dots, \curve_m}$ in the plane $\plane{K}$, all distinct and irreducible, with equations $\curve_k \colon f_k(\varx) =0$, where we include the coordinate axes $L_i \colon x_i = 0$ for $i=0,1,2$. The \emph{complement of the arrangement in the torus} $X \coloneqq \plane{K} \setminus \cup_{D\in \arrCur} D$ is a very affine variety, and it follows from e.g. \cite[Lemma 6.1]{HackingKeelTevelev} that its intrinsic torus $\intTor{X}$ is isomorphic to $\torusCoords{m+2}$ generated by the $f_i$ and coordinates $x_1,\, x_2$ for $\torusCoords{2}\subset \plane{K}$. We will consider the tropicalization of $X$ as a subvariety of its intrinsic torus.

In \cite[Section 5]{CuetoGeomTrop} and \cite[Proposition 5.3]{SturmfelsTevelev}, the following algorithm for computing the tropicalization of $X$ using geometric tropicalization is presented. One iteratively blows up the intersection points of curves in the arrangement until one obtains a combinatorial normal crossing divisor, from which the cones of the tropicalization are extracted. More precisely, one proceeds as follows:
\begin{enumerate}
    \setcounter{enumi}{-1}
    \item Let $\arrCur$ be an arrangement of curves on a surface $\Xbar$, and let $X \coloneqq \Xbar \setminus \cup_{D\in \arrCur} \; D$. Proceed to \eqref{algo:step_1}.
    \item\label{algo:step_1} If there is a point $\varp \in \Xbar$ such that $\varp$ is contained in at least three of the curves of $\arrCur$, proceed to \eqref{algo:step_2}, otherwise proceed to \eqref{algo:step_3}.
    \item\label{algo:step_2} Blow up the point $\varp$, replace $\Xbar$ by the blown up surface $\Xbar'$, and replace the arrangement $\arrCur$ with the arrangement $\arrCur' \coloneqq \set{D'}_{D\in \arrCur} \cup \set{E}$ of curves on $\Xbar'$, where $E$ is the exceptional divisor of the blow up, and $D'$ is the strict transform of $D$. Note that $X'\coloneqq \Xbar' \setminus \cup_{C \in \arrCur'} C$ is isomorphic to $X\subset \Xbar$, and so replace $X$ with $X'$. Return to \eqref{algo:step_1}.
    \item\label{algo:step_3} Since no three curves of $\arrCur$ have a common intersection point, we may view $\arrCur$ as a combinatorial normal crossing divisor compactifying $X$ into the surface $\Xbar$.
\end{enumerate}

Using the above procedure, by \cref{cnc_trop} the tropicalization $\trop(X)$ can be equipped with a fan structure consisting of one ray for each element $D$ of $\arrCur$, and one two-dimensional cone between the rays for $D$ and $D'$ if these two curves intersect in $S$. Moreover note that one may ignore the restriction on the characteristic of $K$ by \cref{cnc_no_char}. Therefore, to complete the description of the tropicalization $\trop(X)$, it remains to find the directions of these rays.

By \cite{HackingKeelTevelev}, page 182, there is a short exact sequence of abelian groups 
\begin{equation}\label{sesPic}
    0 \to \Pic(\Xbar)^\vee \to \bigoplus_{D_i \in \partial X} \Z\cdot D_i \to (\Osheaf(X)^*/K^*)^\vee \to 0,
\end{equation}
where the first map is dual to the map $D_i \mapsto [D_i]\in \Pic(\Xbar)$, and the second is the dual of the map defined on generators by $f_i \mapsto \div(f_i)=D_i$. This allows us to compute the directions of the rays associated to the divisors solely in terms of  the intersection properties of the curves $D$ in $\arrCur$. We illustrate this with the following examples.

\begin{figure}
    \begin{minipage}{.5\textwidth}
        \centering
        \begin{tikzpicture}
            \node (A1) at (-1,0)   {$L_1$};
            \node (A2) at (3,0)   {};

            \node (B1) at (0,3)   {$L_2$};
            \node (B2) at (0,-1)  {};
            
            \node (C1) at (-1,3)  {$L_3$};
            \node (C2) at (3,-1)  {};
            
            \draw (A1) -- (A2);
            \draw (B1) -- (B2);
            \draw (C1) -- (C2);

            \coordinate (X) at (0,0);
            \coordinate (Y) at (2,0);
            \coordinate (Z) at (0,2);
            
            \tkzDefPoint(0,0){X}\tkzDefPoint(2,0){Y}\tkzDefPoint(0,2){Z}
            
            \tkzCircumCenter(X,Y,Z)\tkzGetPoint{O}
            \tkzDrawArc(O,X)(Z)
            \tkzDrawArc(O,Z)(Y)

            \node at (-0.7,1) {$\curve$};
        \end{tikzpicture}
    \end{minipage}%
    \begin{minipage}{.5\textwidth}
        \centering
        \begin{tikzpicture}

            \draw (-0.5,0) -- (4.5,0);
            \draw (-0.5,1) -- (4.5,1);
            \draw (1.5,2) -- (4.5,2);
            \draw (-0.5,3) -- (2.5,3);
            
            \draw (0,-0.5) -- (0,3.5);
            \draw (2,0.5) -- (2,3.5);
            \draw (4,-0.5) -- (4,3.5);
            
            \node at (-0.7,0) {$\widetilde{L_1}$};
            \node at (4.7,2) {$\widetilde{L_3}$};
            \node at (-0.7,3) {$\widetilde{L_2}$};
            \node at (-0.7,1) {$\widetilde{\curve}$};
            
            \node at (2,3.7) {$E_1$};
            \node at (0,-0.7) {$E_3$};
            \node at (4,-0.7) {$E_2$};
        \end{tikzpicture}
    \end{minipage}
            
            
    \captionof{figure}{Curve arrangement from \cref{linesandconic} and the compactifying divisor.}
    \label{fig:linesandconic}
\end{figure}

\begin{example}\label{linesandconic}
    
    Consider the arrangement of curves consisting of the coordinate axes $L_1,L_2,L_3 \subset \plane{K}$, as well as the conic $\curve$ passing through the three points of intersection $L_i\cap L_j$, as shown in \cref{fig:linesandconic}. The complement of the arrangement $X\coloneqq \plane{K} \setminus (L_1\cup L_2 \cup L_3 \cup \curve)$ is not compactified by a simple normal crossing divisor in $\plane{K}$, so we must blow up the three points $L_i\cap L_j$. This yields a configuration including three exceptional divisors $E_1$, $E_2$ and $E_3$, also displayed in \cref{fig:linesandconic}.
    
    Blowing up the plane $\plane{K}$ in the three points $L_i \cap L_j$ to compactify $X$ into $\Xbar$ gives the Picard group $\Pic(\Xbar)\cong \langle H \rangle \oplus_{i=1}^3 \langle e_j\rangle$, where $H$ is the pullback of the line class in $\Pic(\plane{K})$. Moreover, the boundary $\partial X$ consists of seven curves, and the intrinsic torus $\Osheaf(X)^* /K^*$ is of dimension three. The short exact sequence from \eqref{sesPic} is then
    $$0 \to \langle H\rangle \oplus_{i=1}^3 \langle e_j\rangle  \xrightarrow{\phi} \oplus_{i=1}^3 \langle \widetilde{L_i}\rangle  \oplus \langle \widetilde{\curve}\rangle \oplus_{i=1}^3 \langle e_j\rangle   \xrightarrow{\pi} \Osheaf(X)^*/K^*,$$
    where the lattice map $\phi$ is given by the matrix
    \begin{equation*}
        \begin{pmatrix}
        1&0&-1&-1\\
        1&-1&0&-1\\
        1&-1&-1&0\\
        2&-1&-1&-1\\
        0&1&0&0\\
        0&0&1&0\\
        0&0&0&1
        \end{pmatrix}.
    \end{equation*}
    To compute the rays corresponding to each divisor $D$ of the boundary $\partial X$, we consider $\Osheaf(X)^*/K^*$ as the cokernel of $\phi$ under the quotient map $\pi$. Somewhat abusing notation, each column yields a relation in $\Osheaf(X)^*/K^*$ as follows
    \begin{align*}
        \pi(E_1)  &= \phantom{\pi(\widetilde{L_1}) +} \pi(\widetilde{L_2}) + \pi(\widetilde{L_3}) + \pi(\widetilde{C}) \\
        \pi(E_2)  &= \pi(\widetilde{L_1}) \phantom{+ \pi(\widetilde{L_2})} + \pi(\widetilde{L_3}) + \pi(\widetilde{C}) \\
        \pi(E_3)  &= \pi(\widetilde{L_1}) + \pi(\widetilde{L_2}) \phantom{+ \pi(\widetilde{L_3})} + \pi(\widetilde{C})
    \end{align*}
    as well as the equality
    $$\pi(\widetilde{L_1})+\pi(\widetilde{L_2})+\pi(\widetilde{L_3})+2\pi(\widetilde{C})=0.$$
    In particular, the rays of the fan corresponding to the divisors are given by the columns of following matrix
    $$\begin{pmatrix}
        -1 & 1 & 0 & 0 & 1 & -1 & 0 \\
        -1 & 0 & 1 & 0 & 1 & 0 & -1 \\
        -2 & 0 & 0 & 1 & 1 & -1 & -1
    \end{pmatrix},$$
    and embedding the dual complex $\Delta_{\partial X}$ as the cone over these rays, we observe that multiple of the rays are contained inside a two-dimensional face, and we are left with only the four rays
    $$\begin{pmatrix}
        -1&0&0&1\\
        0&-1&0&1\\
        -1&-1&1&1
    \end{pmatrix},
    $$
    which gives a fan with support equal to the tropicalization of $X$.
\end{example}

\section{Axioms for abstract arrangements of curves}\label{section:arroids}
In light of the description of the tropicalization of an arrangement of curves in \cref{section:curve_arr}, we now propose an axiomatization generalizing arrangements of curves. We show that for a given arrangement, the axiomatization data extracted from the arrangement is sufficient to compute its tropicalization in its intrinsic torus.

\subsection{Axioms}
Let $\groundset\coloneqq \set{1, \dots, n}$ be a finite set, called the \emph{ground set} or \emph{divisors}. Let $r$ be an integer, either $1$ or $2$, called the \emph{(reduced) rank}. There is a \emph{degree} function $d\colon \groundset \to \Z_{>0}$, and for each element $i\in \groundset$, we will use the notation $d_i \coloneqq d(i)$. Let $\points$ be a multiset of subsets of $\groundset$, called the \emph{points}. 
Each point $\pbf \in \points$ is equipped with an \emph{intersection multiplicity function} $\mul{\pbf} \colon \pbf^r \to \Z$, satisfying the two conditions:
\begin{itemize}
    \item for $(i_1,\dots, i_r) \in \pbf^r$ a vector, there are bounds $$ 1\leq \mul{\pbf}(i_1,\dots, i_r) \leq \max_{k \in \set{i_1,\dots, i_r}}(d_k),$$ 
    \item $\mul{\pbf}$ is invariant under permutation of the coordinates. 
\end{itemize}
The multiset of points $\points$ satisfies the \emph{Bézout property} if for all vectors $(i_1,\dots, i_r) \in \groundset^r$, there are exactly $d_{i_1} d_{i_2} \cdots d_{i_r}$ points of $\points$ containing $\set{i_1,\dots, i_r}$, when counting with intersection multiplicity, i.e.
$$\sum_{\pbf \supseteq \set{i_1,\dots, i_r}} \mul{\pbf}(i_1,\dots, i_r) = d_{i_1} d_{i_2} \cdots d_{i_r},$$
where the sum is taken over all points $\pbf \in \points$. Let $\mulSet \coloneqq \set{\mul{\pbf} \mid \pbf \in \points}$ denote the set of multiplicity functions associated to $\points$.
\begin{definition}
    An \emph{arroid} is a tuple $(\groundset, d, \points, \mulSet)$ satisfying the Bézout property.
\end{definition}
The data contained in an arroid of rank two records the intersection properties of an arrangement of curves in the plane. More generally, any loop-free rank three matroid gives rise to a rank two arroid, as shown in the following example.
\begin{example}
    Let $M$ be a loop-free matroid of rank three on a ground set $E\coloneqq \set{1,\dots,n}$, with rank function $r\colon 2^E \to \N$. We construct an arroid $\arroid$ as follows. The ground set $\groundset$ of $\arroid$ is taken to be $E$, equipped with the degree function $d$ taking constant value $1$. The multiset of points $\points$ of this arroid is exactly the set of rank two flats of the matroid, where the multiplicity functions are also constant of value $1$. For the Bézout property take a pair $(i,j)$ of elements in the ground set, and note that since $M$ is loop-free, $\set{i}$ is a flat. By the flat partition property, there is exactly one rank three flat containing $j$, thus there is exactly one rank three flat, i.e. point of the arroid, containing $\set{i,j}$ as required.
\end{example}
By the above construction, any matroid has a corresponding arroid. In particular, nonrealizable matroids are not realizable as arroids either, when requiring the specified curves to satisfy the conditions imposed by the degree function and multiplicity functions, i.e. that they must be lines intersecting in one point. However, when removing these restrictions, it may be possible to find curve arrangements having the intersection properties of the original unrealizable matroids. For nonrealizability in higher degrees, see \cref{ex:non-pascal}.

For more general examples, consider the following:
\begin{example}
    We return to \cref{linesandconic}. Let $\groundset = \set{1,2,3,4}$, where $i=1,2,3$ corresponds to the line $L_i$, and $4$ corresponds to the conic $\curve$. The degree function $d\colon \groundset \to \Z_{>0}$ is given by $d_i =1$ for $i<4$ and $d_4=2$. The multiset of points is 
    $\points = \set{\set{1,2,4},\set{1,3,4},\set{2,3,4}},$ corresponding to the set of intersection points of the lines and conic. The Bézout property of $\points$ follows directly from the Bézout theorem for $\plane{K}$.
\end{example}
\begin{figure}
    \begin{minipage}[b]{.5\textwidth}
        \centering
        \begin{tikzpicture}
            \node (A1) at (-1,0)   {$L_1$};
            \node (A2) at (3,0)   {};

            \node (B1) at (0,3)   {$L_2$};
            \node (B2) at (0,-1)  {};
            
            \node (C1) at (-1,3)  {$L_3$};
            \node (C2) at (3,-1)  {};
            
            \draw (A1) -- (A2);
            \draw (B1) -- (B2);
            \draw (C1) -- (C2);

            \coordinate (X) at (-0.2,-0.2);
            \coordinate (Y) at (2.3,0);
            \coordinate (Z) at (0,2.3);

            \tkzDefPoint(-0.2,-0.2){X}\tkzDefPoint(2.3,0){Y}\tkzDefPoint(0,2.3){Z}
            
            \tkzCircumCenter(X,Y,Z)\tkzGetPoint{O}
            \tkzDrawArc(O,X)(Z)
            \tkzDrawArc(O,Z)(Y)


            \node at (-1,1) {$\curve$};
        \end{tikzpicture}
    \end{minipage}
    \captionof{figure}{Curve arrangement from \cref{ex:generic_lines_and_conic}.}
    \label{fig:generic_lines_and_conic}
\end{figure}
\begin{example}\label{ex:generic_lines_and_conic}
    Consider now the generic arrangement of three lines and a conic displayed in \cref{fig:generic_lines_and_conic}. As before, let $\groundset = \set{1,2,3,4}$, where $i=1,2,3$ corresponds to the line $L_i$, and $4$ corresponds to the conic $\curve$, with the degree function $d\colon \groundset \to \Z_{>0}$ given by $d_i =1$ for $i<4$ and $d_4=2$. The multiset structure of the points is now more evident as
    $$\points = \set{\set{1,2},\set{1,3},\set{2,3},\set{1,4},\set{1,4},\set{2,4},\set{2,4},\set{3,4},\set{3,4}},$$
    which reflects that each pair of lines meets in a single point, and the conic meets each line in two points.
\end{example}
In general, given an arrangement of curves $\arrCur=\set{C_1,\dots, C_n}$ on $\plane{K}$, we can create an arroid of rank two as follows. Let $\groundset\coloneqq \set{1,\dots,n}$ and the function $d$ takes $i$ to the degree of the curve $d_i$. For any subset $\pbf\subseteq \arrCur$ such that $C_{I}= \cap_{i \in I} \; C_i \neq \emptyset$, let $\pbf_{I_1}, \dots, \pbf_{I_k}$ be the components of $C_{I}$. For each $\pbf_{I_j}$, we include $\pbf$ as an element in $\points$ with multiplicity function $\mul{\pbf}$ given by the intersection multiplicities of the curves at the component $\pbf_{I_j}$. Thus the point $\pbf$ may appear multiple times with potentially different multiplicity functions. Summarizing, we have show the following proposition.
\begin{proposition}
   Given a curve arrangement $\arrCur$, one can construct a rank two arroid, denoted $\arroid_\arrCur$.
\end{proposition}
Using the above construction, one can also find further examples of non-realizable arroids, for instance by specifically omitting relations that must occur given specified curves and arrangements from classical geometry, as in the next example. 
\begin{figure}
    \begin{tikzpicture}[line cap=round,line join=round,>=triangle 45,x=0.5cm,y=0.5cm]
        \clip(-3,-1.5) rectangle (5.5,5.5);
        \draw [rotate around={0.16511706311106467:(1.1118385024615836,2.0014546547928176)}] (1.1118385024615836,2.0014546547928176) ellipse (3.992296433236598 and 2.904678292888097);
        \draw [domain=-6.533837424058322:21.57561093560145] plot(\x,{(--7.289096040106482-4.892332358437676*\x)/3.7813469178192216});
        \draw [domain=-6.533837424058322:21.57561093560145] plot(\x,{(--16.334590594434452-4.083647648608613*\x)/5.666401996454482});
        \draw [domain=-6.533837424058322:21.57561093560145] plot(\x,{(-2.298279412146368-5.513594205850183*\x)/-0.9935098714566144});
        \draw [domain=-6.533837424058322:21.57561093560145] plot(\x,{(--19.465117260981522-4.8662793152453805*\x)/3.539970293402479});
        \draw [domain=-6.533837424058322:21.57561093560145] plot(\x,{(-0.5732962341811039-5.246498183760384*\x)/-3.438213015221819});
        \draw [domain=-6.533837424058322:21.57561093560145] plot(\x,{(--12.076032390511482-5.407868002984644*\x)/-0.7897879289979861});
        \draw [dash pattern=on 3pt off 3pt,domain=-8:21] plot(\x,{(-5.17234869016573-0.022300850608855782*\x)/-2.584828201125548});
        \begin{scriptsize}

            \draw (-1.6,0.45) node {$C$};
            \draw (-1.6,3.6) node {$1$};
            \draw (-0.5,3.6) node {$2$};
            \draw (0.5,3.6) node {$3$};
            \draw (1.6,3.6) node {$4$};
            \draw (2.5,3.6) node {$5$};
            \draw (3,3.6) node {$6$};
                


            \draw (-0.15,2.4) node {$\mathbf{a}$};
            \draw (1.2,2.4) node {$\mathbf{b}$};
            \draw (2.45,2.4) node {$\mathbf{c}$};
            
            \draw (-2,2.4) node {$L$};
        \end{scriptsize}
    \end{tikzpicture}
    \captionof{figure}{Diagrammatic representation of the nonrealizable arroid from \cref{ex:non-pascal}, with the ``line'' $L$ in dashes.}
    \label{fig:non-pascal}
\end{figure}
\begin{example}[Non-Pascal arroid]\label{ex:non-pascal}
    Let $\arroid$ be the arroid encoding parts of the data of a (self-crossing) hexagon inscribed in a conic. Its data, most easily intuited from \cref{fig:non-pascal}, is given by the ground set $\groundset \coloneqq \set{C, 1, \dots, 6, L}$, with $d(C)=2$, $d(L)=1$ and $d(i)=1$ for all $i$, the multiplicity functions are all constant taking value $1$, with multiset of points $\points$ consists of the points named $\mathbf{a}=\set{1,3,L}$, $\mathbf{b}=\set{2,5,L}$, $\mathbf{c}=\set{4,6,L}$, as well as the points $\set{1,2,C}$, $\set{3,4,C}$, $\set{5,6,C}$, $\set{3,5,C}$, $\set{1,6,C}$, $\set{2,4,C}$, $\set{2,3}$, $\set{4,5}$,$\set{1,5}$, and $\set{2,6}$.

    Pascal's theorem, also known as the hexagrammum mysticum theorem (the reader may prefer \emph{hexagrammum nonmysticum} as the name of the above arroid), states that for curves satisfying the above data, the three points $\mathbf{a}$, $\mathbf{b}$ and $\mathbf{c}$ must be colinear. By omitting $L$ from $\mathbf{c}$ in the arroid, it can no longer be realized in the projective plane. This example also generalizes the non-Pappus matroid, in light of the generalization of Pappus theorem by Pascal's theorem, and one might construct further nonrealizable examples by working with the Cayley--Bacharach theorem.
\end{example}

An arroid of rank one should be viewed as formalizing the process of restricting an arrangement of curves to considering the points of the arrangement on one of the given curves, while recording exactly which curves pass through a given point. In fact we may formalize this in the notion of an \emph{arroid contraction}, contracting a rank $r=2$ arroid to a rank $r=1$ arroid, as defined below. 
\begin{definition}
    Let $\arroid = (\groundset, d, \points, \mulSet)$ be a rank $2$ arroid, and $i\in \groundset$ a divisor. The \emph{contraction} $\arroid/i$ of $\arroid$ is the arroid $\arroid/i \coloneqq (\groundset \setminus \set{i}, d/i , \points/i,  \mulSet^{\arroid/i})$, where:
    \begin{itemize}
        \item   $d/i (j) \coloneqq d_i \cdot d(j)$,
        \item $\points/i$ is the multiset $\set{\pbf\setminus \set{i} \mid i \in \pbf }_{\pbf \in \points}$,
        \item The multiplicity functions, now of rank $r=1$, are the functions $$\mul{\pbf\setminus \set{i}}^{\arroid/i} \colon \pbf \setminus \set{i} \to \Z$$ given by mapping an element $j\in \pbf \setminus \set{i}$ to $\mul{\pbf}(i,j)$, i.e. using the multiplicity function of the original points $\pbf$.
    \end{itemize}
    The Bézout condition for $\arroid/i$ follows from that for $\arroid$.
\end{definition}

Furthermore, in the context of arrangements of curves, one may always consider the arrangement obtained by removing one of the curves. In the arroid context, we generalize this notion in the form of an \emph{arroid deletion}. 
\begin{definition}
    The \emph{deletion} $\arroidDeletion{i}$ of and arroid $\arroid$ is the arroid $\arroidDeletion{i} \coloneqq (\groundset \setminus \set{i}, d, \points\setminus i, \mulSet^{\arroidDeletion{i}})$, where:
    \begin{itemize}
        \item $d$ takes the same values as for $d$ of $\arroid$,
        \item $\points \setminus i$ is the multiset consisting of the points $\pbf$ of $\points$ not containing $i$, as well as points $\pbf\setminus \set{i}$ where $\pbf \neq \set{i,j}$ for some $j\in \groundset$.
        \item Since each point $\mathbf{q}$ of $\arroidDeletion{i}$ corresponds to a specific point $\pbf$ of $\arroid$, we define the multiplicity function 
        $$\mul{\mathbf{q}}^{\arroidDeletion{i}} \colon \mathbf{q}^2 \to \Z$$
        as given by $\mul{\mathbf{q}}^{\arroidDeletion{i}}(j,k)=\mul{\pbf}^\arroid(j,k)$ for all $j,k\in \mathbf{q}$.
    \end{itemize}
\end{definition}




Finally, we describe a certain type of arroid, which is inspired by the case of an arrangement of curves where all the curves intersect pairwise transversely.
\begin{definition}\label{def:transversal_arroid}
    A \emph{transversal arroid} is an arroid such that for all $\pbf \in \points$, the intersection multiplicity $\mul{\pbf}$ is constant taking value $1$.
\end{definition}
Note that the arroid deletion of a transversal arroid is itself transversal.

\subsection{The fan of a transversal arroid}\label{section:arroid_fan}
Let $\arroid=(\groundset, d, \points, \mulSet)$ be a transversal arroid. We will now construct a weighted rational polyhedral fan $\Sigma_\arroid \subset \R^{|\groundset|}/\langle(d_1,\dots,d_n)\rangle$ associated to a transversal arroid. 
This is inspired by the construction of the Bergman fan as the cone over the order complex of matroids as done in \cite{ArdilaKlivans}.

\begin{theorem}\label{theorem:arroid_fan_tropical_variety}
    To each transversal arroid $\arroid$, one can associate an \emph{arroid fan} $\Sigma_\arroid$ which is a balanced tropical variety.
\end{theorem}

We will prove this by first considering and constructing such a fan for rank one arroids, then give a similar construction for rank two.

We begin by indroducing some notation. Let $[w]$ denote the class of the vector $w\in \R^{|\groundset|}$ in the quotient space $\R^{|\groundset|}/\langle(d_1,\dots,d_n)\rangle$ , and $e_1,\dots, e_n$ be the standard basis of $\R^{|\groundset|}$. Denote by $\points^u$ the underlying set of the multiset $\points$, i.e. discarding the number of occurrences of each point $\pbf \in \points$ and recording it only once in $\points^u$.

\subsubsection{Rank one arroid}\label{sec:rank-one-arroid-fan}
First consider $\arroid$ an arroid of rank one, and define $\Sigma_\arroid \subset \R^{|\groundset|}/\langle(d_1,\dots,d_n)\rangle$ to be the following weighted one-dimensional rational polyhedral fan. For each point $\pbf \in \points^u$, let $v_\pbf \coloneqq \sum_{j \in \pbf} [e_j]$ be a ray of $\Sigma_\arroid$. The weight function $\omega\colon \Sigma_1 \to \Z$ is defined by taking $v_\pbf$ to $w(\pbf)$, where $w(\pbf)$ is the number of times $\pbf$ is repeated in the multiset $\points$. This fan is \emph{tropical}, in the sense that it satisfies the balancing condition in codimension one \cite[Definition 5.7]{BIMS14}, as we have 
$$\sum_{\pbf\in \points^u} \omega(\pbf) v_\pbf = \sum_{\pbf\in \points} \sum_{j \in \pbf} [e_j] = [\sum_{j \in \groundset} \sum_{\substack{\pbf\in \points \\ j \in \pbf}} e_j] = [\sum_{j \in \groundset} d_j e_j] =0,$$
where we used the Bézout property and the transversality to conclude that $\sum_{j \in \pbf} 1= d_j$.

\subsubsection{Rank two arroids}\label{sec:rank-two-arroid-fan}
Now we turn to arroids $\arroid$ of rank two. We must show that:
\begin{enumerate}
    \item one can associate a fan given the data of a rank two arroid, and
    \item that the associated fan is tropical.
\end{enumerate}

We define the fan as follows. For each divisor $i \in \groundset$, the cone $\rho_i \coloneqq \cone([e_j])$ is a ray of $\Sigma_\arroid$. For each point $\pbf\in \points^u$, the cone $\rho_{\pbf} \coloneqq \cone(v_\pbf)$, with $v_\pbf \coloneqq \sum_{i \in \pbf} [e_j]$, is also a ray of $\Sigma_\arroid$. The two-dimensional cones of $\Sigma_\arroid$ are the following. For each pair $i\in \groundset$ and $\pbf \in \points$ with $i \in \pbf$, there is a two-dimensional cone $\sigma_{i,\pbf}\coloneqq\cone(\rho_i, \rho_\pbf)$ in $\Sigma_\arroid$. The weight function $\omega\colon \Sigma_2 \to \Z$ is defined by taking $\sigma_{i,\pbf}$ to $w(\pbf)$, where $w(\pbf)$ is the number of times $\pbf$ is repeated in the multiset $\points$. 

\begin{lemma}\label{lemma:rank-two-arroid-fan-is-tropical}
    The fan $\Sigma_\arroid$ is tropical.
\end{lemma}
\begin{proof}
    Now we must verify balancing at each ray. There are two types of rays to verify: the $\rho_\pbf$ and the $\rho_i$.
    Picking a ray $\rho_\pbf$, we must verify that the sum $\sum_{\sigma_{i,\pbf} \supface \rho_{\pbf}} \omega(\sigma_{i,\pbf}) v_{\sigma_{i,\pbf}/\rho_{\pbf}}$ is a vector contained in $\rho_{\pbf}$, where $v_{\sigma_{i,\pbf}/\rho_{\pbf}}$ together with the primitive vector $v_\pbf$ of $\rho_{\pbf}$ generates $\sigma_{i,\pbf}$. One such $v_{\sigma_{i,\pbf}/\rho_{\pbf}}$ is precisely $[e_j]$, the primitive vector of $\rho_i$, and so we have 
    $$\sum_{\sigma_{i,\pbf} \supface \rho_{\pbf}} \omega(\sigma_{i,\pbf}) v_{\sigma_{i,\pbf}/\rho_{\pbf}} = \sum_{i \in \pbf} w(\pbf) [e_j] = w(\pbf) [\sum_{i\in \pbf} e_j] = w(\pbf) [v_\rho],$$
    which is contained in $\rho_\pbf$. The balancing condition for the rays $\rho_i$ follows by using the Bézout condition in the following manner.
    We pick $v_{\sigma_{i,\pbf}/\rho_i} = [v_\pbf]$, then
    \begin{align*}
        \sum_{\sigma_{i,\pbf} \supface \rho_i} \omega(\sigma_{i,\pbf}) v_{\sigma_{i,\pbf}/\rho_i} = \sum_{\substack{\pbf \ni i \\ \pbf \in \points^u}} w(\pbf) [v_\pbf] = \sum_{\substack{ \pbf \ni i \\ \pbf \in \points}} [\sum_{j \in \pbf} e_j] = \sum_{\substack{j \in \groundset \\ j \neq i}} \sum_{\substack{\pbf \supset \set{i,j} \\ \pbf \in \points}}  [e_j] + \sum_{\substack{\pbf \ni i \\ \pbf \in \points}}[e_j].
    \end{align*}
    Now by the Bézout condition, each vector $e_j \neq e_i$ appears exactly $d_i d_j$ times, so the first term is equivalent to a multiple of $[e_i]$ in the quotient $\R^{|\groundset|}/\langle(d_1,\dots,d_n)\rangle$. This implies in particular that the whole sum is equal to some multiple of $[e_i]$, and hence contained in $\rho_i$. 
\end{proof}

\begin{proof}[Proof of \cref{theorem:arroid_fan_tropical_variety}]
    By the constructions made in \cref{sec:rank-one-arroid-fan} and \cref{sec:rank-two-arroid-fan}, we have associated a fan $\Sigma_\arroid$ for each transversal arroid $\arroid$, which we call the \emph{arroid fan} of $\arroid$. Moreover, by the discussion in \cref{sec:rank-one-arroid-fan} and \cref{lemma:rank-two-arroid-fan-is-tropical}, these fans are tropical.
\end{proof}

\begin{remark}\label{remark:indec_conn_codim_one}
    It follows from the above construction that the fan of a transversal arroid of rank one or two is necessarily \emph{connected through codimension one} in the sense of \cite[Definition 3.3.4]{MaclaganSturmfels}.
\end{remark}

\subsection{Operations on transversal arroid fans}
In this section, we describe operations that can be performed using the concept of arroid fans. We begin by describing the reduced star at the rays of the fan $\Sigma_\arroid$ an arroid $\arroid$ of rank two in terms of arroid fans. 
\begin{proposition}\label{prop:arroid_star}
    Let $\Sigma_\arroid$ be an arroid fan, and $i\in \groundset$ correspond to the ray $\rho_i$ of $\Sigma_\arroid$. Then the reduced star $\Star{\Sigma_\arroid}{\rho_i}$ is equal to the arroid fan $\Sigma_{\arroid / i}$ of the arroid contraction $\arroid / i$.
\end{proposition}
\begin{proof}
    This follows by considering both definitions. The fan $\Sigma_{\arroid / i}$ is contained in the associated space $\R^{|\groundset \setminus \set{i}|}/\langle (d_1,\dots, \hat{d_i}, \dots, d_n) \rangle$, which can also be viewed as the space $(\R^{|\groundset|}/\langle (d_1,\dots, d_n) \rangle ) / \langle e_i \rangle$ that contains $\Star{\Sigma_\arroid}{\rho_i}$. Moreover, the cones of both these fans correspond exactly to the projection of the cones of $\Sigma$ containing $\rho_i$, i.e. the points $\pbf\in \points$ containing $i$.
\end{proof}
Similarly, it follows from the definitions that, given the arroid $\arroid$ of an arrangement of curves $\arrCur$, and considering the arrangement obtained by removing one of the curves, say numbered $i$, from the arrangement $\arrCur$, the corresponding arroid will be the deletion $\arroidDeletion{i}$.

Next, we show that the fan of a transversal arroid may be constructed inductively using tropical modifications along any of its arroid deletions. This is inspired by the construction of the Bergman fan of a matroid through tropical modifications described in \cite[Proposition 1.1.34]{ShawThesis}, as well as the \emph{shellability} property in \cite{AminiPiquerezFans}. 
\begin{proposition}\label{prop:arroid_tropical_modif}
    Let $\arroid=(\groundset, d, \points, \mulSet)$ be a transversal arroid, and $\arroidDeletion{i}$ one of its arroid deletions. Then $\Sigma_\arroid = \tm(\Sigma_{\arroidDeletion{i}}, \Sigma_{\arroid / i})$ is a tropical modification of the fan of the arroid deletion $\arroidDeletion{i}$ along the fan of the arroid contraction $\arroid / i$.
\end{proposition}
\begin{proof}
    We will show that the preimage of a point in the fan $\Sigma_{\arroid}$ under the coordinate projection $\pi \colon \R^{|\groundset|}/\langle(d_1,\dots, d_n)\rangle \to \R^{|\groundset\setminus i|}/\langle(d_1,\dots, \hat{d_i} ,\dots ,d_n) \rangle$ is either a half-line in the $e_i$ direction, or a point. Then one may obtain a piecewise integer affine function $g$ defining the tropical modification by taking $g(p) = \pi^{-1}(p)$ for the points $p\in \Sigma_{\arroidDeletion{i}}$ whose preimage $\pi^{-1}(p)$ are single points. 

    To show that $\pi^{-1}(p)$ is either a point or half-line in $\Sigma_{\arroid}$, we analyze the cones of the latter. The cones of $\Sigma_{\arroid}$ have one of three forms: 
    \begin{itemize}
        \item $\cone(e_j, e_j+e_i)$ for some $j\in \groundset$, corresponding to a point $\pbf$ which disappears in $\arroidDeletion{i}$,
        \item $\cone(e_j, e_\mathbf{a} + e_i)$, corresponding to a point $\pbf = \mathbf{a}\cup {i} \in \points$ from which $i$ is removed in $\arroidDeletion{i}$, or 
        \item $\cone(e_j, e_\pbf)$ for some $\pbf\in \points$ not containing $i$.
    \end{itemize}

    Any point $p$ of the fan $\Sigma_{\arroidDeletion{i}}$ is contained inside $\cone(e_k,e_\pbf)$ for some $k$ and $\pbf\in \points \setminus i$ with $k \in \pbf$. We distinguish multiple cases:
    \begin{itemize}
        \item If $p=\mathbf{0}$ is the origin, then $\pi^{-1}(p)=\cone(e_i)$ is a half-ray of the fan $\Sigma_{\arroid}$.
        \item If $p = \alpha e_k + \beta e_\pbf$ with $\alpha,\beta >0$, and $\pbf \cup \set{i}$ is a point of $\arroid$, the preimage $\pi^{-1}(p)$ is $p' = \alpha e_j + \beta (e_\pbf+ e_i)$ in $\Sigma_{\arroid}$.
        \item If $p= \alpha e_k$ then either $\set{i,k}\in \points$, and then $\pi^{-1}(p)= \set{a e_j +b (e_j+e_i) \mid a+b=\alpha; \; a,b \geq 0}$ is a half-ray or $\set{i,k} \not\in \points$ and then $\pi^{-1}(p)=\alpha e_k$.
        \item If $p = \beta e_\pbf$ then either $\pbf\cup \set{i}$ is a point of $\arroid$, and the preimage is a half.line, or it is not so that $\pi^{-1}(p)=\beta e_\pbf$.
    \end{itemize}
    For any of $p$, the fiber is either a half-line or a point, and thus $\Sigma_{\arroid}$ is a tropical modification of $\Sigma_{\arroid \setminus i}$.
    
    Moreover, it follows from the construction that the divisor along which the modification is performed is the reduced star $\Star{\Sigma_\arroid}{\rho_i}$, which is $\Sigma_{\arroid/ i}$ by \cref{prop:arroid_star}, so that we have $\Sigma_\arroid = \tm(\Sigma_{\arroidDeletion{i}}, \Sigma_{\arroid / i})$. Moreover, the weights of $\Sigma_\arroid$ are then given in terms of the weights of the rays of $\Sigma_{\arroid / i}$, which is compatible with the construction of both fans as above.
\end{proof}

\subsection{Tropicalization}\label{section:arroid_tropicalization}

We now show that the data gathered in the transversal arroid of an arrangement of curves is sufficient to compute its tropicalization. It is in fact sufficient to note that, all the data used in the description of the tropicalization given in \cref{section:curve_arr} is recorded in the arroid. Our aim is instead to show that, for an arrangement where all intersections are transverse, the fan described in \cref{section:arroid_fan} is a fan structure on the tropicalization of the complement of the arrangement.

First, to insure that the arrangement of curves $\arrCur$ in the plane $\plane{K}$ is in fact very affine, we suppose that it contains at least three lines intersecting generically, so that they can be mapped to the coordinate axes by a projective transformation. We will say that such an arrangement is \emph{very affine}. If all the curves intersect pairwise transversely, we will say that the arrangement $\arrCur$ is \emph{transverse}. In this case, the associated arroid is transversal because the multiplicity functions at each point of the arroid $\arroid_\arrCur$ is given by the intersection multiplicities of the curves of $\arrCur$ at the corresponding point. Since all the curves intersect pairwise transversely, the multiplicity function $m_\pbf$ is constant of value $1$, hence the arroid is transversal.

\begin{theorem}\label{theorem:transverse_very_affine_curve_fan}
    Let $\arrCur$ be a transverse very affine arrangement of curves in the plane $\plane{K}$. Then the tropicalization $\trop(\arrCompCur)$ of the complement is supported on the fan of the associated transversal arroid $\arroid_\arrCur$.
\end{theorem}
\begin{proof}
    We show that the arroid fan is supported on the tropicalization of the complement of the arrangement. Blowing up $\plane{K}$ in any point where more than $3$ of the curves in $\arrCur$ meet, we obtain a simple normal crossing divisor $D$ which compactifies the complement $\arrCompCur\coloneqq \plane{K} \setminus \cup_{C\in \arrCur} \; C$. Therefore, using the map to the intrinsic torus, \cref{theorem:snc_trop} gives us that the tropicalization  is equal to the fan whose rays are $\cone([\val_{D}])$ in $N_\R$ for each irreducible component $D$ of the simple normal crossing divisor $\divisor$, and whose two-dimensional cones are $\cone([\val_{D}], [\val_{D'}])$ for boundary divisors $D$ and $D'$ such that $D\cap D' \neq \emptyset$. 
    
    There are two types of irreducible components of $\divisor$: the strict transforms of the curves $C$ in $B$, whose divisorial valuations are such that the associated vectors $[\val_C]$ is a standard basis vector of $N_\R$, and the exceptional divisors $E$ of the blown-up loci $\pbf$, for which the valuation $\val_E$ computes the order of vanishing at $\pbf$. 
    Therefore, the vector $[\val_E]$ is the sum $\sum_{i_k} [\val_{D_{i_k}}]$ for the $D_{i_k}$ which intersect at $\pbf$. 
    In particular, it follows from this description of the rays that the tropicalization is exactly equal to the fan of the arroid $\arroid_\arrCur$.
\end{proof}

\section{Tropical homology manifold arroid fans}\label{section:THM_arroid}
In this section, we study when a transversal arroid fan is a tropical homology manifold. We will first show that all tropical Borel--Moore homology of a transversal arroid fan is concentrated in top-degree, and hence torsion-free. Then we will use an Euler characteristic argument to show that being a tropical homology manifold reduces to a balancing condition. This property is a main component in understanding when the complement of an arrangement of curves is cohomologically tropical, see \cref{theorem:cohom_trop}.

\subsection{Tropical cohomology of transversal arroid fans}
For a transversal arroid, the main structural result is that the tropical Borel--Moore homology of the associated fan is concentrated in top-degree. More precisely, we have the following lemma.
\begin{lemma}\label{lemma:BM_vanishing}
    Let $\arroid$ be a transversal arroid. Then $H_{p,q}^{BM}(\Sigma_\arroid)=0$ for $q\neq 2$, $\forall p$.
\end{lemma}
\begin{proof}
    By \cref{prop:arroid_tropical_modif}, each arroid fan $\Sigma_\arroid$ is the tropical modification of any of its deletions, by adding a divisor to the ground set. Hence it suffices to show that the property holds for the modification of an arroid fan. Let $\Sigma_{\arroid}$ be an arroid fan, and $\tm(\Sigma_\arroid, \tropdiv(\phi))$ a modification of $\Sigma_{\arroid}$. As pointed out in the proof of \cite[Proposition 5.5]{Lefschetz11}, there is a short exact sequence of complexes 
    \begin{align*}
        0 \to C_{p,\bullet}^{BM}(\tropdiv(\phi))&\to C_{p,\bullet}^{BM}(\ctm(\Sigma_\arroid, \tropdiv(\phi)))\\ &\to C_{p,\bullet}^{BM}(\tm(\Sigma_\arroid, \tropdiv(\phi))) \to 0,
    \end{align*}
    and moreover, by \cref{lemma:closed_modif_original_fan}, the long exact sequence in homology takes the form 
    \begin{align*}
        \cdots \to H_{p,q}^{BM}(\tropdiv(\phi))\to H_{p,q}^{BM}(\Sigma_\arroid) &\to H_{p,q}^{BM}(\tm(\Sigma_\arroid, \tropdiv(\phi))) \\ &\to H_{p,q-1}^{BM}(\tropdiv(\phi)) \to \cdots.
    \end{align*}
    By induction we have $H_{p,q}^{BM}(\Sigma_\arroid)=0$ for $q\neq 2$, and similarly $H_{p,q}^{BM}(\tropdiv(\phi))=0$ for $q\neq 1$, since $\tropdiv(\phi)$ is one-dimensional. Therefore the result follows by exactness.
\end{proof}
\begin{lemma}[{\cite[Lemma 2.2.7]{ShawThesis}}]\label{lemma:ses_tropical_co_hom}
    Let $\Sigma\subseteq \R^{n}$ be a tropical fan, with its tropical modification $\tm(\Sigma_\arroid, \tropdiv(\phi))$ along the tropical divisor $\tropdiv(\phi)$. Then for each $p=1,\dots, \dim \Sigma$, we have short exact sequences
    \[\begin{tikzcd}[column sep=5pt,row sep=tiny]
        0 & {H_{p-1,0} (\tropdiv(\phi))} & {H_{p,0} (\tm(\Sigma,\tropdiv(\phi)))} & {H_{p,0} (\Sigma)} & {0,} \\
        0 & {H^{p,0} (\Sigma)} & {H^{p,0} (\tm(\Sigma,\tropdiv(\phi)))} & {H^{p-1,0} (\tropdiv(\phi))} & {0,}
        \arrow[from=2-1, to=2-2]
        \arrow[from=1-1, to=1-2]
        \arrow["\gamma", from=1-2, to=1-3]
        \arrow["\delta", from=1-3, to=1-4]
        \arrow[from=1-4, to=1-5]
        \arrow["{\delta^\vee}", from=2-2, to=2-3]
        \arrow["{\gamma^\vee}", from=2-3, to=2-4]
        \arrow[from=2-4, to=2-5]
    \end{tikzcd}\]
    dual to each other, where $\gamma$ is the map $w\mapsto w \wedge e_{n+1}$ and $\delta$ is the map $v_1 \wedge \cdots \wedge v_p \mapsto \pi(v_1) \wedge \cdots \wedge \pi(v_p)$, where $\pi \colon \R^{n+1} \to \R^{n}$ is the projection onto the first $n$ components.
\end{lemma}
\begin{proof}
    The  \cite[Lemma 2.2.7]{ShawThesis} is stated for tropical modifications of matroid fans, however the proof is applicable to this more general context. We summarize it briefly. Duality between diagrams follows from the modules involved being $\Q$-vector spaces, so it suffices to prove exactness in the first diagram. To prove surjectivity of $\delta$, it suffices for instance to find preimages of the standard basis vectors. The injectivity of $\gamma$ is also straightforward. To complete the proof, it therefore suffices to observe that $\ker(\delta) = \im(\gamma)$, since $w\in \ker(\delta)$ if and only if $w = w' \wedge e_{n+1}$, but $w'\wedge e_{n+1} \in H_{p,0} (\tm(\Sigma,\tropdiv(\phi)))$ if and only if $w' \in H_{p-1,0} (\tropdiv(\phi))$, so that $w=\gamma(w')$.
\end{proof}
\begin{lemma}\label{lemma:closed_modif_original_fan}
    Let $\Sigma\subseteq \R^{n}$ be a tropical fan, and $\ctm(\Sigma_\arroid, \tropdiv(\phi))$ its closed tropical modification along the tropical divisor $\tropdiv(\phi)$. Then the isomorphism $H_{p,q}^{BM}(\Sigma_\arroid) \cong H_{p,q}^{BM}(\ctm(\Sigma_\arroid, \tropdiv(\phi)))$ holds for all $p$ and $q$. 
\end{lemma}
\begin{proof}
    The \cite[Lemma 5.7]{Lefschetz11} holds for arbitrary tropical modifications by the above \cref{lemma:ses_tropical_co_hom}, and thus the proof given for \cite[Proposition 5.6]{Lefschetz11} generalizes to the present context.
\end{proof}

Note that the above strategy is applicable also in the case of $\Z$-coefficients for tropical homology, which shows that there is no torsion in the tropical Borel--Moore homology of the fan of a transversal arroid. Moreover, it follows that the fan of a transversal arroid is a tropical homology manifold if and only if it is uniquely balanced along each ray.
\begin{theorem}\label{theorem:arroid_THM_is_local}
    Let $\Sigma_\arroid$ be the arroid fan of a transversal arroid. Then $\Sigma_\arroid$ is a  tropical homology manifold if and only if it is uniquely balanced along each of its rays.
\end{theorem}
\begin{proof}
    By \cref{lemma:BM_vanishing}, the tropical Borel--Moore homology groups are such that $H_{p,q}^{BM}(\Sigma_\arroid)=0$ for $q\neq 2$, for all $p$. Moreover, for each cone $\gamma$ of $\Sigma$, the reduced star fan $\Star{\Sigma}{\gamma}$ has $H_{p,q}^{BM}(\Star{\Sigma}{\gamma})=0$ for $q\neq 1$, for all $p$, since $\Star{\Sigma}{\gamma}$ is one-dimensional. Therefore, the result follows by \cite[Theorem 5.10]{TPDspace}. 
\end{proof}
Finally, we conclude with an argument showing that for arroids, the deletion operation preserves the tropical homology manifold property of the fan.
\begin{lemma}\label{lemma:arroid_contraction_THM}
    Let $\arroid$ be a transversal arroid such that the arroid fan $\Sigma_\arroid$ is a tropical homology manifold. Then for any arroid deletion $\arroid \setminus i$, the corresponding arroid fan $\Sigma_{\arroid \setminus i}$ is a tropical homology manifold.
\end{lemma}
\begin{proof}
    Since $\Sigma_\arroid$ is a tropical homology manifold, the reduced star along the ray corresponding to $i$, which by \cref{prop:arroid_star} is the arroid fan $\Sigma_{\arroid/i}$ of the contraction $\arroid / i$, satisfies Tropical Poincaré Duality. By \cref{prop:arroid_tropical_modif}, we have that $\Sigma_\arroid = \tm(\Sigma_{\arroid \setminus i}, \Sigma_{\arroid / i})$, so that for each $p=0,1,2$, \cite[Diagram 5.6]{Lefschetz11} takes the form
    \[\begin{tikzcd}
        0 & {H^{p,0}(\Sigma_{\arroidDeletion{i}})} & {H^{p,0}(\Sigma_\arroid)} & {H^{p,0}(\Sigma_{\arroid / i})} & 0 \\
        0 & {H_{2-p,2}^{BM}(\Sigma_{\arroidDeletion{i}})} & {H_{2-p,2}^{BM}(\Sigma_{\arroid})} & {H_{2-p,2}^{BM}(\Sigma_{\arroid/i})} & {0.}
        \arrow[from=1-1, to=1-2]
        \arrow[from=1-2, to=1-3]
        \arrow[from=1-3, to=1-4]
        \arrow[from=1-4, to=1-5]
        \arrow["{\frown [\Sigma_{\arroid \setminus i}]}", from=1-2, to=2-2]
        \arrow[from=2-2, to=2-3]
        \arrow[from=2-3, to=2-4]
        \arrow[from=2-4, to=2-5]
        \arrow["{\frown [\Sigma_{\arroid / i}]}", from=1-4, to=2-4]
        \arrow["{\frown [\Sigma_\arroid]}", from=1-3, to=2-3]
        \arrow[from=2-1, to=2-2]
    \end{tikzcd}\]
    By assumption both the middle and rightmost vertical arrows are isomorphisms, and therefore so is the leftmost vertical arrow. Therefore $\Sigma_{\arroid \setminus i}$ satisfies tropical Poincaré duality, and so it remains to show that each reduced star $\Star{\Sigma_{\arroid \setminus i}}{\gamma}$ of a cone $\gamma \in \Sigma_{\arroid \setminus i}$ satisfies tropical Poincaré duality.
    In light of \cref{lemma:BM_vanishing}, the result follows by applying \cite[Proposition 5.7]{TPDspace}.
\end{proof}

\subsection{Unique balancing of arroid rays}\label{section:unique_balancing_rays_arrangement}
We now begin a preliminary investigation into when the rays of an arroid fan are uniquely balanced. In light of \cref{theorem:arroid_THM_is_local}, establishing such conditions would be equivalent to describing which arroid fans are tropical homology manifolds.

\begin{lemma}
    Let $\Sigma_\arroid$ be an arroid fan and $\pbf$ a point. Then the ray $\rho_\pbf$ is uniquely balanced.
\end{lemma}
\begin{proof}
    This follows since the equality 
    $$\sum_{i \in \pbf} \alpha_i [e_i] = \alpha [\nu_\rho]$$
    holds if and only if $\alpha= \alpha_i$ for all $i$, where $\nu_\rho=\sum_{j \in \pbf} e_j$. Therefore the only balancing along such an edge is given by the weight $w(\pbf)$ for each facet containing $\rho_\pbf$, as in \cref{section:arroid_fan}.
\end{proof}
\begin{corollary}
    The only rays of an arroid fan for which unique balancing may fail, are the rays $\rho_i$ corresponding to divisors $i\in \groundset$. 
\end{corollary}
For such a divisor ray, a balancing corresponds to a set of weights $\alpha_\pbf$ for each $\pbf$ containing $i$, such that the equality
\begin{equation}\label{eq:balancing_ray_divisor}
    \sum_{\pbf \ni i } \alpha_\pbf [\nu_\rho] = \alpha [e_i]
\end{equation}
holds, where the sum is taken over the set of unique points $\points^u$. Using the balancing already exhibited in \cref{section:arroid_fan}, we may increment both sides such that $\alpha_\pbf \geq 0$ for all $\pbf$.

We will now describe conditions on an arrangement of lines and conics which guarantee the unique balancing property along the rays corresponding to the lines. We begin with a certain connectedness notion for arrangements.
\begin{definition}
    Let $\arrCur$ be a transverse very affine arrangement of curves in $\plane{K}$, and $C$ a curve of the arrangement. A \emph{cluster} of curves $\cluster\subset \arrCur$ on $C$ is a set of curves that for any two curves $C_a, C_b \in \cluster$, there is a sequence $C_a = C_{1},C_{2},\dots, C_{m} = C_b$ such that $C_{i}\cap C_{i+1}$ contains a point of $C$.
\end{definition}

\begin{example}
    \begin{figure}
        \centering
        \begin{tikzpicture}[scale=0.25]
            \clip(-9,-6) rectangle (20,6);
            \draw  (-2,0.8333333333333334) circle (2.1666666666666665cm);
            \draw  (2,0.8333333333333334) circle (2.1666666666666665cm);
            \draw  (0,-0.7333333333333333) circle (4.066666666666666cm);
            \draw  (14,0.8333333333333334) circle (2.1666666666666665cm);
            
            \draw  (-8,0) -- (18,0);
            \draw  (6,-10) -- (6,12);
            \draw  (12,-10) -- (12,12);
            \draw  (16,-10) -- (16,12);
            \begin{scriptsize}
            \draw[color=black] (-7, 1) node {$f$};
            \draw[color=black] (-2, -2) node {$c$};
            \draw[color=black] (2, -2) node {$d$};
            \draw[color=black] (0, 4) node {$e$};
            \draw[color=black] (7,4) node {$g$};
            \draw[color=black] (11.5, -2) node {$h$};
            \draw[color=black] (14, 4) node {$k$};
            \draw[color=black] (16.5, -2) node {$i$};
            \end{scriptsize}
            \end{tikzpicture}
        \captionof{figure}{A line with three clusters}
        \label{fig:cluster}
    \end{figure}

    \cref{fig:cluster} displays an arrangement where, for the line $f$, the sets of clusters are $\set{e,c,d}, \set{k,h,i}$ and the singleton $\set{g}$. The set $\set{g,d}$ is not a cluster for $f$.
\end{example}
A cluster may be a subset of another cluster, giving an ordering. For a chosen curve $C$, we may divide the points of the arrangement on $C$ into the maximal clusters in which they are intersection points. For a line $L$ in an arrangement of lines and conics, the following condition on the clusters of $L$ implies unique balancing of the corresponding ray of the arroid fan.
\begin{lemma}\label{lemma:arroid_balance_line_ray}
    Let $\arrCur$ be a transverse very affine arrangement of lines and conics in $\plane{K}$, and let $\arroid$ be the arroid of the arrangement. Let $i\in \groundset$ be a divisor corresponding to a line $L$. Suppose that each cluster of $L$ contains a line. Then the arroid fan $\Sigma_\arroid$ is uniquely balanced along the ray $\rho_i$.
\end{lemma}
\begin{proof}
    We will show that for any cluster, the coefficients $\alpha_\pbf$ are all equal, for any point $\pbf$ on $L$ of the cluster. For each cluster, let $\pbf_\ell$ be the point on $L$ containing the line, and $\alpha_{\pbf_\ell}$ the corresponding coefficient in equation \eqref{eq:balancing_ray_divisor}. Now for any of the conics $C$ passing through this point, the Bézout condition, along with balancing implies that the other intersection point $\pbf'$ of $C$ with $L$ must have coefficient $\alpha_{\pbf'}=\alpha_{\pbf_\ell}$. By iterating through all pairs in the cluster, all coefficients of the points of the cluster are equal. Moreover, as this is true for all clusters of $L$, the balancing condition forces the coefficients of different clusters to be equal, making the balancing unique.
\end{proof}

The unique balancing of rays associated with conics is seemingly more subtle. We first introduce specific types of clusters, from which we derive coarse criterion for unique balancing to occur.

\begin{definition}
    Let $\arrCur$ be a transverse very affine arrangement of lines and conics in $\plane{K}$, and $C$ a conic of the arrangement. A conic $C'$ of a cluster $\cluster$ is a \emph{source} if all four points of $C'\cap C$ are joined by at least five lines of the arrangement, and it called a \emph{reservoir} if it is joined by at least four. For $C''$ another conic of the cluster, an \emph{aqueduct} on $C$ between $C'$ and $C''$ is a line $L$ of the arrangement joining a pair of intersection points in $C'\cap C$ and $C'' \cap C$. A cluster containing at least one source conic, with all other conics being reservoir conics, with a chain of aqueducts linking back to the source, is called a \emph{balancing supply system}.
\end{definition}
The sense in the above naming is that, for a source, all weights $\alpha_\pbf$ of the associated intersection points must be equal. If an aqueduct connects a source to a reservoir, then the weights of the intersection points at the reservoir must be equal to that of the source, and this equality is then spread further to all other reservoirs linked by aqueducts; i.e. if a cluster is a balancing supply system, all the weights of any balancing are equal.
\begin{figure}
    \centering
    \begin{tikzpicture}[scale=0.5]
        \clip(-9.5,-5.5) rectangle (9.5, 5.5);
        \draw [rotate around={90:(-5,0)}, line width=0.6pt] (-5,0) ellipse (3.6055512754639873cm and 2cm);
        \draw [rotate around={0:(0,0)}, line width=0.6pt] (0,0) ellipse (7.280109889280522cm and 2cm);
        \draw [domain=-10:11, line width=0.6pt] plot(\x,{(--10.59839157514747--1.214485704937359*\x)/3.7113613433324577});
        \draw [domain=-10:11, line width=0.6pt] plot(\x,{(-10.59839157514747-1.214485704937359*\x)/3.7113613433324577});
        \draw [domain=-10:11, line width=0.6pt] plot(\x,{(-14.34027328167274-2.3611800873767006*\x)/3.7113613433324577});
        \draw [domain=-10:11, line width=0.6pt] plot(\x,{(--14.34027328167274--2.3611800873767006*\x)/3.7113613433324577});
        \draw [rotate around={90:(5,0)}, line width=0.6pt] (5,0) ellipse (3.6055512754639873cm and 2cm);
        \draw [domain=-10:11, line width=0.6pt] plot(\x,{(--10.598391575147453-1.2144857049373559*\x)/-3.7113613433324524});
        \draw [domain=-10:11, line width=0.6pt] plot(\x,{(--10.598391575147451-1.2144857049373554*\x)/3.7113613433324524});
        
        \draw [line width=0.6pt] (-3.263190056416981,-7) -- (-3.263190056416981,7);
        \draw [line width=0.6pt] (-7,-7) -- (-7,7);
        \draw [line width=0.6pt]  (7,-7) -- (7,7);
        \draw [line width=0.6pt]  (3.263190056416983,-7) -- (3.263190056416983,7);
        \draw[domain=-10:11,line width=0.6pt] plot(\x,{(-0--3.5756657923140596*\x)/-6.526380112833964});
        \begin{scriptsize}
        \draw [fill=black] (-6.974551399749439,0.5733471912196708) circle (1pt);
        \draw [fill=black] (-3.263190056416981,1.7878328961570298) circle (1pt);
        \draw [fill=black] (-6.974551399749439,-0.5733471912196708) circle (1pt);
        \draw [fill=black] (-3.263190056416981,-1.7878328961570298) circle (1pt);
        \draw [fill=black] (6.974551399749435,-0.5733471912196741) circle (1pt);
        \draw [fill=black] (3.263190056416983,-1.78783289615703) circle (1pt);
        \draw [fill=black] (3.263190056416983,1.7878328961570296) circle (1pt);
        \draw [fill=black] (6.974551399749435,0.5733471912196741) circle (1pt);

        \node[label={$C''$}] at (2.5,0) {};
        \node[label={$C'$}] at (-2.5,0) {};
        \node[label={$C$}] at (0,-2) {};
        \node[label={$L$}] at (0,0) {};
        \end{scriptsize}
        \end{tikzpicture}
    \captionof{figure}{A conic with balancing supply cluster}
    \label{fig:cluster_balancing}
\end{figure}
\begin{example}
    In \cref{fig:cluster_balancing}, we consider the conic $C$, which has a unique maximal cluster, comprised of all the lines together with the two additional conics $C'$ and $C''$. Here $C'$ is a source, $C''$ is a reservoir, and the line $L$ is an aqueduct. Therefore the unique maximal cluster is a balancing supply system.
\end{example}
\begin{lemma}\label{lemma:arroid_balance_conic_ray}
    Let $\arrCur$ be a transverse very affine arrangement of lines and conics in $\plane{K}$, and let $\arroid$ be the arroid of the arrangement. Let $i\in \groundset$ be a divisor corresponding to a conic $C$. If all maximal cluster of $C$ contain conics and are balancing supply system, then the arroid fan $\Sigma_\arroid$ is uniquely balanced along the ray $\rho_i$.
\end{lemma}
\begin{proof}
    We follow the strategy of weight propagation suggested above. Let $\cluster$ be a maximal cluster, and $C'$ a source. Since the arrangement is transverse, there are four intersection points $\pbf_1,\pbf_2,\pbf_3$ and $\pbf_1$ in $C\cap C'$, and given that there are 5 lines passing through these points, they yield four distinct terms $\alpha_\pbf [\nu_\pbf]$ in the left hand sum \eqref{eq:balancing_ray_divisor}. The balancing condition implies that the $\alpha_\pbf$ must all be equal, as otherwise the vectors $[e_j]$ corresponding to the lines do not appear an equal number of times. Next, let $C''$ be a reservoir, and $\pbf\in C'' \cap C$ a point where there is an aqueduct $L$ between $C'$ and $C''$. Then if the weights $\alpha_{\pbf_i} = \alpha$ are fixed, since the coordinate $[e_L]$ must appear $2\alpha$ times, we have that $\alpha_{\pbf}=\alpha$ as well. Then for the same reason, lines passing through two intersection points of $C'' \cap C$ have components appearing $2\alpha$ times, thus all points of $C'' \cap C$ must have the same weight. This reasoning will then also apply to any other reservoir connected to $C''$ by an aqueduct, and therefore the weights of all points in a cluster which is a balancing supply system are equal. Moreover, the balancing condition implies that the weights of two distinct clusters must agree. Therefore there is a unique balancing along the ray $\rho_i$.
\end{proof}

For a transverse very affine arrangement of lines and conics $\arrCur$, a straightforward condition for all the maximal clusters of a conic $C$ to be balancing supply systems, is simply that for any two other conics $C'$ and $C''$, all possible lines between points in $C' \cap C$ and $C'' \cap C$ are part of the arrangement.

\section{Cohomologically tropical arrangements}
In this section, we study two properties of transverse very affine arrangements of curves. First, we show that under certain conditions, the complement of the arrangement is \emph{wunderschön} in the sense of \cite[Definition 1.2]{AAPS23}, which is primarily a restriction on the Mixed Hodge structure of the cohomology of this complement and of the involved curves. Next we use the study of arroid fans which are tropical homology manifold as discussed in \cref{section:THM_arroid}, to show that for certain arrangements, the complement is \emph{cohomologically tropical} in the sense of \cite[Definition 1.1]{AAPS23}, i.e. that one may compute all the cohomology of the complement using only the tropical variety.

Let $\arrCur$ be a transverse very affine arrangement of curves in the plane $\plane{\C}$, and $\arrCompCur$ its complement, which we identify as a subvariety of its intrinsic torus $\intTor{\arrCompCur}$. The associated arroid $\arroid = \arroid_\arrCur$ then yields the tropicalization of $\arrCompCur\subseteq \intTor{\arrCompCur}$ as the support of the arroid fan $\Sigma_\arroid$.
\begin{proposition}\label{prop:wunderschon_arrangement}
    Let $\arrCur$ be a transverse very affine arrangement of non-singular rational curves in $\plane{\C}$ such that no intersection point of the arrangement contains exactly the same curves, then $\arrCompCur$ is a wunderschön variety.
\end{proposition}
\begin{proof}
    Consider the toric variety $\C\P_{\Sigma_\arroid}$ associated to the (unimodular) fan $\Sigma_\arroid$ and the closure $\arrCompCurBar \subseteq \C\P_{\Sigma_\arroid}$. Each cone $\sigma$ of $\Sigma_\arroid$ gives a torus orbit $\torus^\sigma \subseteq \C\P_{\Sigma_\arroid}$, and let $\arrCompCur^\sigma \coloneqq \arrCompCurBar \cap \torus^\sigma$. To show that $\arrCompCur$ is wunderschön, we must first show that each $\arrCompCur^\sigma$ is non-singular. 
    
    For the minimal cone $\mincone$ of $\Sigma_\arroid$, corresponding to the central vertex, $\arrCompCur^\mincone$ is the complement of the curves of $\arrCur$, hence is non-singular. For a ray $\rho_i$ of $\Sigma_\arroid$ corresponding to a divisor $i \in \groundset$ of the arroid, the intersection $\arrCompCur^{\rho_i}$ is the curve of $\arrCur$ corresponding to $i$, which is non-singular. For a ray $\rho_\pbf$ corresponding to a point $\pbf\in \points$ of the arroid, the intersection $\arrCompCur^{\rho_\rho}$ is the exceptional curve of the blow-up at the corresponding point, which is also non-singular. Moreover for each cone $\sigma_{i,\pbf} = \cone(e_i, e_\pbf)$ of $\Sigma_\arroid$, since no intersection point of the arrangement contains exactly the same curves, the intersection $\arrCompCur^{\sigma_{i,\pbf}}$ is the intersection of the exceptional curve corresponding to $\pbf$ with the curve corresponding to $i$, which is exactly one point, hence non-singular.

    Next, to show that $\arrCompCur$ is wunderschön, we must verify that for each $\arrCompCur^\sigma$, the mixed Hodge structure on $H^k (\arrCompCur^\sigma)$ is pure of weight $2k$. This follows for the two-dimensional cones from the assumption on points, and follows for each one-dimensional cone as it corresponds to a punctured non-singular rational curve. Finally, for $\arrCompCur^\mincone=\arrCompCur$, we consider Deligne's weight spectral sequence for the mixed Hodge structure \cite[§7]{Hodge1}. All the homology of $\arrCompCurBar$ is generated by that of $\plane{\C}$ together with that of the exceptional divisors of the blown-up points. Therefore, there is no odd degree cohomology of $\arrCompCurBar$, and since all the boundary divisors are rational, all odd rows of the spectral sequence are identically zero. Then it remains to show that the cohomology of the following two complexes, corresponding to the only non-zero even rows of the weight spectral sequence, is concentrated on the left
    \[\begin{tikzcd}[column sep=small,row sep=tiny]
        0 & {\oplus_{p}\; H^0(p)} & {\oplus_{D}\; H^2(D)} & {H^4(\arrCompCurBar)} & {0,} \\
        & 0 & {\oplus_D \; H^0(D)} & {H^2(\arrCompCurBar)} & {0.}
        \arrow[from=1-1, to=1-2]
        \arrow[from=1-4, to=1-5]
        \arrow[from=2-2, to=2-3]
        \arrow[from=2-3, to=2-4]
        \arrow[from=2-4, to=2-5]
        \arrow[from=1-2, to=1-3]
        \arrow[from=1-3, to=1-4]
    \end{tikzcd}\]
    Here the $D$'s are the components of the boundary divisors compactifying $\arrCompCur$ to $\arrCompCur$, and the $p$'s correspond to their pairwise intersections. Since the dual graph of the compactifying divisor is connected, the top complex only has cohomology on the left by \cite[Theorem 3.1]{Hacking}. Moreover, the homology $H_2(\arrCompCurBar)$ is generated by the homology of $\plane{\C}$ along with the homology of the exceptional divisors of the blown-up points \cite[p. 474]{GriffithsHarris}, so the map in the lower complex is surjective, and gives the cohomology $\mathrm{Gr}_2^W H^2(\arrCompCur)=0$. We therefore have $\mathrm{Gr}_{i}^W H^k(\arrCompCur)$ is $0$ if $i\neq 2k$, i.e. the mixed Hodge structure is pure of weight $2k$.
\end{proof}

A transverse very affine arrangements $\arrCur$ of non-singular rational curves in $\plane{\C}$, i.e. of lines and conics, such that that no intersection point of the arrangement contains exactly the same curves, will be called \emph{simple}. For simple arrangements, we relate the rational cohomology of their complements with the tropical cohomology of their tropicalizations, using the notion of cohomologically tropical varieties.

We give a brief description of the \emph{cohomologically tropical} property for varieties and their tropicalizations, referring to \cite{AAPS23} for the original definition and greater details. For $X \subseteq \torus$ a subvariety of a torus, and $\Trop(X)$ its tropicalization, picking a unimodular fan $\Sigma$ supported on $\Trop(X)$ yields a compactification $\Xbar$ of $X$ inside the toric variety $\C\P_\Sigma$, as well as a compactification $\overline{\Trop(X)}$ of $\Trop(X)$ inside the tropical toric variety $\mathbb{T}\P_\Sigma$. There is a map $\tau^* \colon H^{\bullet}(\overline{\Trop(X)}) \to H^{\bullet}(\Xbar)$, relating the $\Q$-coefficient tropical cohomology of $\overline{\Trop(X)}$ to the rational cohomology ring of $\Xbar$ (see \cite{AAPS23} for more on this map). This map can also be defined for each of the reduced stars and the corresponding projections of the variety $X$. If all these maps are isomorphisms, $X$ is said to be \emph{cohomologically tropical}. 
For simple arrangements, we have the following result regarding which have cohomologically tropical complements.
\begin{theorem}\label{theorem:cohom_trop}
    Let $\arrCompCur$ be the complement of a simple arrangement $\arrCur$. Then $\arrCompCur$ is cohomologically tropical if and only if the corresponding arroid fan $\Sigma_{\arroid_\arrCur}$ is uniquely balanced along each of its rays.
\end{theorem}
\begin{proof}
    By \cref{prop:wunderschon_arrangement}, this complement is necessarily wunderschön. Therefore, the equivalence described in \cite[Theorem 6.1]{AAPS23} reduces to an equivalence between $\arrCompCur$ being cohomologically tropical and $\trop(\arrCompCur)=\Sigma_{\arroid_\arrCur}$ being a tropical homology manifold. The result then follows from \cref{theorem:arroid_THM_is_local}.
\end{proof}

\section{Maximal subvarieties}\label{section:maximal}
In this section, we will study the question of maximality, in the sense of the Smith-Thom inequality \eqref{eq:smith-thom}, for arrangements of curves using the arroid abstraction we have developed so far. For a simple arrangement $\arrCur$ of curves, i.e. consisting of lines and conics, defined over the reals, with all intersections being real, we will now use tropical homology manifolds to give a condition for maximality of the complement of the arrangement. Along with the two next lemmas, this will give examples of varieties satisfying conditions (a), (b) and (c) of \cite[p. 3]{AmbrosiManzaroli}, as described in the introduction.
\begin{lemma}\label{lemma:torsion-free}
    Let $\arrCur$ be an arrangement of curves containing at least one line in $\plane{\C}$, and $\arrCompCur$ its complement. Then the homology groups $H_1(\arrCompCur(\C);\Z)$ and $H_2(\arrCompCur(\C);\Z)$ are torsion-free, and $H_k(\arrCompCur(\C);\Z)=0$ for $k>2$.
\end{lemma}
\begin{proof}
    Since the arrangement contains at least one line, it is affine, therefore the homology group $H_1(\arrCompCur)$ is torsion-free, see e.g. \cite[Corollary 4.1.4]{DimcaHypersurfaces}. Adapting similar methods, one may consider the long exact sequence in homology for the pair $(\plane{\C}, \arrCompCur)$, which yields
    \[\begin{tikzcd}[column sep = 9pt]
        \cdots & {H_3(\plane{\C};\Z)} & {H_3(\plane{\C},\arrCompCur(\C);\Z)} & {H_2(\arrCompCur(\C);\Z)} & {H_2(\plane{\C};\Z)} & {\cdots}
        \arrow[from=1-1, to=1-2]
        \arrow[from=1-2, to=1-3]
        \arrow[from=1-3, to=1-4]
        \arrow[from=1-4, to=1-5]
        \arrow[from=1-5, to=1-6]
    \end{tikzcd}\]
    Since $H_3(\plane{\C};\Z)=0$ and $H_2(\plane{\C};\Z)=\Z$, it suffices to show that $H_3(\plane{\C},\arrCompCur(\C);\Z)$ is torsion-free. Let $T$ be a closed tubular neighborhood of the curve $\curve = \cup_{C \in \arrCur} C$, with boundary $\partial T$. By excision, $H_3(\plane{\C},\arrCompCur(\C);\Z) \cong H_3(T, \partial T;\Z)$, and by Lefschetz duality \cite[p. 254]{HatcherAlgTop}, $H_3(T, \partial T;\Z)\cong H^1 (T;\Z)$. This last group is torsion-free since $H^1$-cohomology is always torsion-free.

    Moreover, since there is at least one line in the arrangement, the variety $\arrCompCur(\C)$ is affine. Therefore all homology $H_k(\arrCompCur(\C);\Z)$ vanishes for $k>2$ by \cite[Theorem 7.1]{MilnorMorseTheory}.
\end{proof}
\begin{lemma}\label{lemma:real_part_no_homology}
    Let $\arrCur$ be a simple arrangement in $\plane{\C}$, with all the curves defined over the reals, and all intersection points being real. Then $H_1(\arrCompCur(\R);\Z)$ is trivial.
\end{lemma}
\begin{proof}
   This follows by induction. By adding the curve $C$ to the arrangement $\arrCompCurDeletion{C}$, the real part $C(\R)$ is an $S^1$ not contained in a component of $\arrCompCurDeletion{C}(\R)$ since all intersection points are real, so no $H_1$ homology is added. 
\end{proof}
\begin{theorem}\label{theorem:THM_maximal}
    Let $\arrCur$ be a simple arrangement in $\plane{\C}$, with all the curves defined over the reals. Suppose that all intersection points in the arrangement are real and $\Trop(\arrCompCur)$ is a tropical homology manifold. Then the complement $\arrCompCur$ is maximal.
\end{theorem}
\begin{proof}
    The arrangement $\arrCur$ consists of smooth rational curves in $\plane{\C}$, of which at least three are lines, so we first consider the smaller arrangement $\arrCurLines$ consisting of only the lines. One may prove that $\arrCompCurLines$ is maximal, see \cite[Corollary 5.95]{OrlikTerao}, so we may proceed by induction on the number of curves in the arrangement.

    Let $\arrCur$ be a simple arrangement of curves, with all intersection points being real, and with the corresponding arroid fan $\Sigma_{\arroid_\arrCur}$ being a tropical homology manifold. By \cref{prop:wunderschon_arrangement}, the arrangement complement $\arrCompCur$ is wunderschön, and so by \cite[Theorem 6.1]{AAPS23}, it is cohomologically tropical. By \cite[Theorem 6.2]{AAPS23}, we have isomorphisms $H^k(\arrCompCur(\C)) \cong H^{k,0}(\Sigma_{\arroid_\arrCur})$, and in particular $b_k(\arrCompCur(\C); \Z/2\Z) = \dim\; H^{k,0}(\Sigma_{\arroid_\arrCur})$, for all $k$ since the homology is torsion-free by \cref{lemma:torsion-free}.

    For any choice of curve $C\in \arrCur$ and corresponding $i\in \groundset$, the reduced star fan of the ray $\rho_i$ is given by $\Sigma_{\arroid / i}$ by \cref{prop:arroid_star}, and is a tropical homology manifold by assumption. It follows from \cref{prop:wunderschon_arrangement} that the punctured curve $C^*$ tropicalizing to $\Sigma_{\arroid /i}$ is wunderschön, and so by \cite[Theorem 6.2]{AAPS23}, $\dim \, H^k(C^*) = \dim \, H^{k,0}(\Sigma_{\arroid / i})$ for all $k$. Moreover, $C^*(\C)$ is homotopic to a wedge of circles, so that its homology is torsion-free, hence $b_k(C^*(\C)) = b_k(C^*(\C); \Z/2\Z)= \dim \; H^{k,0}(\Sigma_{\arroid /i})$ for all $k$.

    Similarly, we study the arrangement $\arrCur \setminus \set{C}$, corresponding to the contracted arroid $\arroid_{\arrCur \setminus i}$. The tropicalization of its complement $\arrCompCurDeletion{C}$ is the arroid fan $\Sigma_{\arroid_{\arrCur \setminus i}}$, which by \cref{lemma:arroid_contraction_THM} is a tropical homology manifold. By induction assumption, the complement $\arrCompCurDeletion{C}$ is maximal, i.e. $b_\bullet(\arrCompCurDeletion{C}(\R))= b_\bullet(\arrCompCurDeletion{C}(\C))$. By \cref{prop:wunderschon_arrangement}, $\arrCompCurDeletion{C}$ is wunderschön, so that by \cite[Theorem 6.1]{AAPS23} it is cohomologically tropical, and therefore $H^{k,0}(\Sigma_{\arroid_{\arrCur \setminus i}}) \cong H^k(\arrCompCurDeletion{C}(\C))$ for all $k$ by \cite[Theorem 6.2]{AAPS23}. Therefore $b_k(\arrCompCurDeletion{C}(\C)) = b_k(\arrCompCurDeletion{C}(\C); \Z/2\Z)= \dim \; H^{k,0}(\Sigma_{\arroid_{\arrCur \setminus i}})$ for all $k$.

    By \cref{lemma:real_part_no_homology}, $b_\bullet (\arrCompCur(\R)) = b_0 (\arrCompCur(\R))$ for any simple arrangement with real intersection points. Moreover, each arc of $C^*(\R)$ increases the number of connected components of $\arrCompCurDeletion{C}(\R)$ by $1$, so counting the number of connected components of $\arrCompCur(\R)$, we have $b_0 (\arrCompCur(\R)) = b_0 (\arrCompCurDeletion{C}(\R)) + b_0 (C^*(\R))$. Note also that the number of punctures of $S^1 \cong C(\R)$ by real points is $b_0(C^*(\R))$, so $C^*(\C)$ is a punctured Riemann sphere, which gives $b_0(C^*(\R)) = b_\bullet (C^*(\C))$. Summarizing, we have 
    \begin{align*}
        b_\bullet(\arrCompCur(\R))= b_0(\arrCompCur(\R)) &=  b_0 (\arrCompCurDeletion{C}(\R)) + b_0(C^*(\R)) \\ &= b_\bullet(\arrCompCurDeletion{C}(\C)) + b_\bullet(C^*(\C)).
    \end{align*}
    
    By \cref{lemma:ses_tropical_co_hom}, we have $\dim\, H^{k-1,0}(\Sigma_{\arroid / i}) + \dim\, H^{k,0}(\Sigma_{\arroid \setminus i}) = \dim\, H^{k,0}(\Sigma_{\arroid})$ for all $k$, which when combined with the equalities between tropical cohomology and singular cohomology described above, yield $b_\bullet(\arrCompCurDeletion{C}(\C)) + b_\bullet(C^*(\C)) = b_\bullet(\arrCompCur(\C))$.
\end{proof}

\begin{figure}
    \centering

    \begin{tikzpicture}[scale=1]
        \clip(-1.75,-1.75) rectangle (3.75,3.75);
        \draw [rotate around={0:(1,1)},line width=0.6pt] (1,1) ellipse (2cm and 1.1547005383792515cm);
        \draw [rotate around={90:(1,1)},line width=0.6pt] (1,1) ellipse (2cm and 1.1547005383792515cm);
        \draw [rotate around={0:(1,1)},line width=0.6pt] (1,1) ellipse (2.5cm and 1.0910894511799618cm);
        \draw [samples=50,domain=-0.99:0.99,rotate around={0:(1,1)},xshift=1cm,yshift=1cm,line width=0.6pt] plot ({0.6546536707079771*(1+(\x)^2)/(1-(\x)^2)},{0.8660254037844386*2*(\x)/(1-(\x)^2)});
        \draw [samples=50,domain=-0.99:0.99,rotate around={0:(1,1)},xshift=1cm,yshift=1cm,line width=0.6pt] plot ({0.6546536707079771*(-1-(\x)^2)/(1-(\x)^2)},{0.8660254037844386*(-2)*(\x)/(1-(\x)^2)});
        \draw [line width=0.6pt,domain=-2.64206528571275:6.025215839019944] plot(\x,{(--4-2*\x)/2});
        \draw [line width=0.6pt,domain=-2.64206528571275:6.025215839019944] plot(\x,{(-0-2*\x)/-2});
        \draw [line width=0.6pt] (2,-1.4626231659845152) -- (2,3.805933092380953);
        \draw [line width=0.6pt] (0,-1.4626231659845152) -- (0,3.805933092380953);
        \draw [line width=0.6pt,domain=-2.64206528571275:6.025215839019944] plot(\x,{(-4-0*\x)/-2});
        \draw [line width=0.6pt,domain=-2.64206528571275:6.025215839019944] plot(\x,{(-0-0*\x)/-2});
        \draw [samples=50,domain=-0.99:0.99,rotate around={90:(1,1)},xshift=1cm,yshift=1cm,line width=0.6pt] plot ({0.5863019699779287*(1+(\x)^2)/(1-(\x)^2)},{0.7237468644557459*2*(\x)/(1-(\x)^2)});
        \draw [samples=50,domain=-0.99:0.99,rotate around={90:(1,1)},xshift=1cm,yshift=1cm,line width=0.6pt] plot ({0.5863019699779287*(-1-(\x)^2)/(1-(\x)^2)},{0.7237468644557459*(-2)*(\x)/(1-(\x)^2)});
        \draw [rotate around={90:(1,1)},line width=0.6pt] (1,1) ellipse (2.5cm and 1.0910894511799618cm);

    \end{tikzpicture}
    \captionof{figure}{A maximal arrangement }
    \label{fig:maximal_arrangement}
\end{figure}

\begin{example}\label{ex:inf_family}
    We illustrate \cref{theorem:THM_maximal} by the example displayed in \cref{fig:maximal_arrangement}. It is constructed by selecting four points, drawing all lines between them, and then selecting any number of non-singular conics passing through those same four points. To see that these arrangements are tropical homology manifolds, it suffices to note that in each intersection point of a line with a conic, multiple other lines pass through the same point, so that all rays associated to lines of the arroid fan are uniquely balanced by \cref{lemma:arroid_balance_line_ray}, and all intersection points of any pair of conics are connected by lines, so that by \cref{lemma:arroid_balance_conic_ray} all rays associated to conics are uniquely balanced. By \cref{theorem:arroid_THM_is_local}, the arroid fan is therefore a tropical homology manifold, hence by \cref{theorem:THM_maximal} each such arrangement is maximal.
\end{example}

Finally, simple real arrangements of curves, all intersecting in real points, with the tropicalization of the complement being a tropical homology manifold, satisfy conditions (a), (b) and (c) of \cite[p. 3]{AmbrosiManzaroli}.
\begin{theorem}\label{theorem:amb_manz_examples}
    Let $\arrCur$ be a simple arrangement of real curves in $\plane{\C}$, with all intersection points being real, and such that the tropicalization $\Trop(\arrCompCur)$, which is supported on the arroid fan $\Sigma_{\arroid_\arrCur}$, is a tropical homology manifold. Then the following four properties are satisfied:
    \begin{enumerate}[label=(\alph*)]
        \item $H^i (\arrCompCur(\R);\Z/2\Z)=0$ for $i\geq 1$,
        \item $\arrCompCur$ is a maximal variety, 
        \item the mixed Hodge structure on $H^i(\arrCompCur(\C);\Q)$ is pure of type $(i,i)$ with cohomology groups $H^i(\arrCompCur(\C);\Z)$ torsion-free for $i\geq 1$, and
        \item $\dim_\Q H^i(\arrCompCur(\C);\Q) = \sum_j \dim_\Q H^{i,j}(\Sigma_{\arroid_\arrCur})$ for each $i\geq 0$.
    \end{enumerate}
\end{theorem}
\begin{proof}
    The first property follows from \cref{lemma:torsion-free} and \cref{lemma:real_part_no_homology}, the second from \cref{theorem:THM_maximal}. The purity of the mixed Hodge structure on rational cohomology is satisfied because $\arrCompCur$ is wunderschön by \cref{prop:wunderschon_arrangement}, and the integer homology is torsion-free by \cref{lemma:torsion-free}. The last property follows from \cref{theorem:cohom_trop}, as $\arrCompCur$ is cohomologically tropical since the tropicalization is a tropical homology manifold.
\end{proof}
In particular, the family of arrangements described in \cref{ex:inf_family} gives examples of the varieties satisfying the conditions of \cite{AmbrosiManzaroli}.

\bibliographystyle{msc} 
\bibliography{bibliography}    

\end{document}